\documentclass{jcmlatex}
\def\JCMvol{xx}
\def\JCMno{x}
\def\JCMyear{200x}
\def\JCMreceived{xxx}
\def\JCMrevised{xxx}
\def\JCMaccepted{xxx}

\usepackage{color}

\begin{document}

\markboth{O.~BURDAKOV, Y.-H.~DAI AND N.~HUANG}{Stabilized Barzilai-Borwein method}

\title{STABILIZED BARZILAI-BORWEIN METHOD}

\author{Oleg Burdakov
\thanks{Department of Mathematics, Link\"oping University, Link\"oping, Sweden \\ Email: oleg.burdakov@liu.se}
\and
 Yu-Hong Dai
\thanks{LSEC, ICMSEC, Academy of Mathematics and Systems Science, Chinese Academy of Sciences, Beijing, China\\ Email: dyh@lsec.cc.ac.cn}
\and
Na Huang\footnote{Corresponding author}
\thanks{Department of Applied Mathematics, College of Science, China Agricultural University, Beijing, China\\ Email: hna@cau.edu.cn}
}

\maketitle
\thispagestyle{empty}

\begin{abstract}
The Barzilai-Borwein (BB) method is a popular and efficient tool for solving large-scale unconstrained optimization problems. Its search direction is the same as for the steepest descent (Cauchy) method, but its stepsize rule is different. Owing to this, it converges much faster than the Cauchy method. A feature of the BB method is that it may generate too long steps, which throw the iterates too far away from the solution. Moreover, it may not converge, even when the objective function is strongly convex. In this paper, a stabilization technique is introduced. It consists in bounding the distance between each pair of successive iterates, which often allows for decreasing the number of BB iterations. When the BB method does not converge, our simple modification of this method makes it convergent. For strongly convex functions with Lipschits gradients, we prove its global convergence, despite the fact that no line search is involved, and only gradient values are used. Since the number of stabilization steps is proved to be finite, the stabilized version inherits the fast local convergence of the BB method. The presented results of extensive numerical experiments show that our stabilization technique often allows the BB method to solve problems in a fewer iterations, or even to solve problems where the latter fails.
\end{abstract}



\section{Introduction}
\label{sec:into}
In this paper, we consider spectral gradient methods for solving the unconstrained optimization problem
\begin{equation}\label{umin}
\min_{x \in R^n} f(x),
\end{equation}
where $f: R^n \rightarrow R^1$ is a sufficiently smooth function. Its minimizer is denoted by $x^*$.
Gradient-type iterative methods used for solving problem \eqref{umin} have the form
\begin{equation}\label{eqitr}
x_{k+1}=x_k-\alpha_kg_k,
\end{equation}
where $g_k=\nabla f(x_k)$ and $\alpha_k>0$ is a stepsize. Methods of this type differ in the stepsize rules which they follow.

We focus here on the two choices of $\alpha_k$ proposed in 1988 by Barzilai and Borwein \cite{Barzilai1988two}, usually refereed to as the BB method. The rationale behind these choices is related to viewing the gradient-type methods as quasi-Newton methods, where $\alpha_k$ in \eqref{eqitr} is replaced by the matrix $D_k=\alpha_kI$. This matrix is served as an approximation of the inverse Hessian matrix. Following the quasi-Newton approach, the stepsize is calculated by forcing either $D_k^{-1}$ (BB1 method) or $D_k$ (BB2 method) to satisfy the secant equation in the least squares sense. The corresponding two problems are formulated as
\begin{equation}\label{ls}
\min_{D=\alpha I}~\|D^{-1}s_{k-1}-y_{k-1}\| \quad \textrm{and} \quad
\min_{D=\alpha I}~\|s_{k-1}-Dy_{k-1}\|,
\end{equation}
where $s_{k-1}=x_k-x_{k-1}$ and $y_{k-1}=g_k-g_{k-1}$. The solutions to these problems are
\begin{equation}\label{bb}
\alpha_k^{BB1}=\frac{s_{k-1}^Ts_{k-1}}{s_{k-1}^Ty_{k-1}} \quad \textrm{and} \quad \alpha_k^{BB2}=\frac{s_{k-1}^Ty_{k-1}}{y_{k-1}^Ty_{k-1}},
\end{equation}
respectively. Here and in what follows, $\|\cdot \|$ denotes the Euclidean vector norm and the induced matrix norm. Other norms used in this paper will be denoted in a different way.

Barzilai and Borwein \cite{Barzilai1988two} proved that their method converges $R$-superlinearly for two-dimensional strictly convex quadratics. Dai and Fletcher \cite{dai2005asymptotic} analyzed the asymptotic behavior of BB-like methods, and they obtained $R$-superlinear convergence of the BB method for the three-dimensional case. Global convergence of the BB method for the $n$-dimensional case was established by Raydan \cite{raydan1993Barzilai} and further refined by Dai and Liao \cite{dai2002r} for obtaining the R-linear rate. For nonquadratic functions, local convergence proof of the BB method with R-linear rate was, first, sketched in some detail by Liu and Dai \cite{LiuDai2001}, and then it was later rigorously proved by Dai et al. \cite{dai2006cyclic}.
Extensive numerical experiments show that the two BB stepsize rules significantly improve the performance of gradient methods (see, e.g., \cite{fletcher2005barzilai,raydan1997barzilai}), both in quadratic and nonquadratic cases.

A variety of modifications and extensions have been developed, such as gradient methods with retards \cite{friedlander1998gradient}, alternate BB method \cite{dai2005projected}, cyclic BB method \cite{dai2006cyclic}, limited memory gradient method \cite{CurtisGuo18} etc. Several approaches were proposed for dealing with nonconvex objective functions, in which case the BB stepsize \eqref{bb} may become negative. In our numerical experiments, we use the one proposed in \cite{dai2015positive}.
The BB method has been extended to solving symmetric and nonsymmetric linear equations \cite{dai2015positive,dai2005analysis}. Furthermore, by incorporating the nonmontone line search by Grippo et al. \cite{Grippo_etal_86}, Raydan \cite{raydan1997barzilai} and Grippo et al. \cite{GrippoSciandrone2002} developed the global BB method for general unconstrained optimization problems. Later, Birgin et al. \cite{birgin2000nonmonotone} proposed the so-called spectral projected gradient method which extends Raydan's method to smooth convex constrained problems. For more works on BB-like methods, see \cite{Birgin2014review,fletcher2005barzilai,yuan2008step} and references therein.

As it was observed by many authors, the BB method may generate too long steps, which throw the iterates too far away from the solution. In practice, it may not converge even for strongly convex functions (see, e.g., \cite{fletcher2005barzilai}). The purpose of this paper is to introduce a simple stabilization technique and to justify its efficiency both theoretically and practically. Our stabilization does not assume any objective function evaluations. It consists in uniformly bounding $\|s_k\|$, the distance between each pair of successive iterates. It should be emphasized that, if the BB method safely converges for a given function, then there is no necessity in stabilizing it. In such cases, the stabilization may increase the number of iterations. In other cases, as it will be demonstrated by results of our numerical experiments, the stabilization may allow for decreasing the number of iterations or even to make the BB method convergent.

Although we focus here on stabilizing the conventional BB method, our approach can directly be combined with the existing modifications of the BB method, where a nonmonotone line search is used.

The paper is organized as follows. In the next section, we present an example of a strictly convex function and show that the BB method does not converge in this case. This contributes to a motivation for stabilizing this method. In the same section, its stabilized version is introduced. In Section \ref{sec:convergence}, a global convergence of our stabilized BB algorithm as well as its R-linear rate of convergence are proved under suitable assumptions. Results of numerical experiments are reported and discussed in Section \ref{sec:numres}. Finally, some conclusions are included in the last section of the paper.

\section{Stabilized Algorithm}
\label{sec:algo}
Before formulating our stabilized algorithm, we wish to begin with a motivation based on presenting an example of a strongly convex function for which we theoretically prove that neither of the BB methods converge.
To the best of our knowledge, no theoretical evidence of BB methods being divergent is available in the literature.

In the review paper by Fletcher \cite{fletcher2005barzilai}, it is claimed that the BB method diverges in practice for certain initial points in the test problem referred to as \textsf{Strictly Convex 2} by Raydan \cite{raydan1997barzilai}, in which
\begin{equation}\label{raydan}
f(x) = \sum_{i=1}^{n} i (e^{x_i} - x_i) / 10.
\end{equation}
This strongly convex function will be used in Section \ref{sec:numres} for illustrating the efficiency of the stabilized algorithm.
Our numerical experiments show that, in this specific case, the failure of the BB method is related to the underflow and overflow effects in the computer arithmetic. We are not acquainted with any theoretical justification of the divergence of the BB method for this or any other functions.

We will present now an instance of a function for which the BB method does not converge in the exact arithmetic. For this purpose, the notation
$$
a=\sqrt{5}-1, \quad b = \sqrt{5}+3, \quad
c_1=\frac{3\sqrt{5}+8}{4},\quad c_2=-\frac{5\sqrt{5}+11}{32},\quad
f(a)=\dfrac{c_1 a^2}{2}+\frac{c_2 a^4}{4}
$$
will be used. Consider the univariate function
\begin{equation}\label{f}
  f(x)=\left\{\begin{array}{ll}
\dfrac{1}{4}(x+a)^2 - (\sqrt{5}+1)(x+a) +f(a), & x< -a,\\[8pt]
\dfrac{ c_1}{2}x^2+\dfrac{c_2}{4}x^4,       & -a \le x\le a,\\[8pt]
\dfrac{1}{4}(x-a)^2 + (\sqrt{5}+1)(x-a) + f(a), & x> a.
\end{array}\right.
\end{equation}
Its first derivative
$$
g(x)=\left\{\begin{array}{ll}
\dfrac{1}{2}(x+a)-\sqrt{5}-1,& x < -a,\\[8pt]
c_1 x +c_2 x^3,     & -a \le x \le a,\\[8pt]
\dfrac{1}{2}(x-a)+\sqrt{5}+1, & x > a
\end{array}\right.
$$
is continuously differentiable, and $g(x)$ is an odd monotonically increasing function (see Figure \ref{fig_g}).
\begin{figure}[thp]
	\vspace{-8.38cm}
	\hspace{-2.2cm} \includegraphics[width=1.25\textwidth]{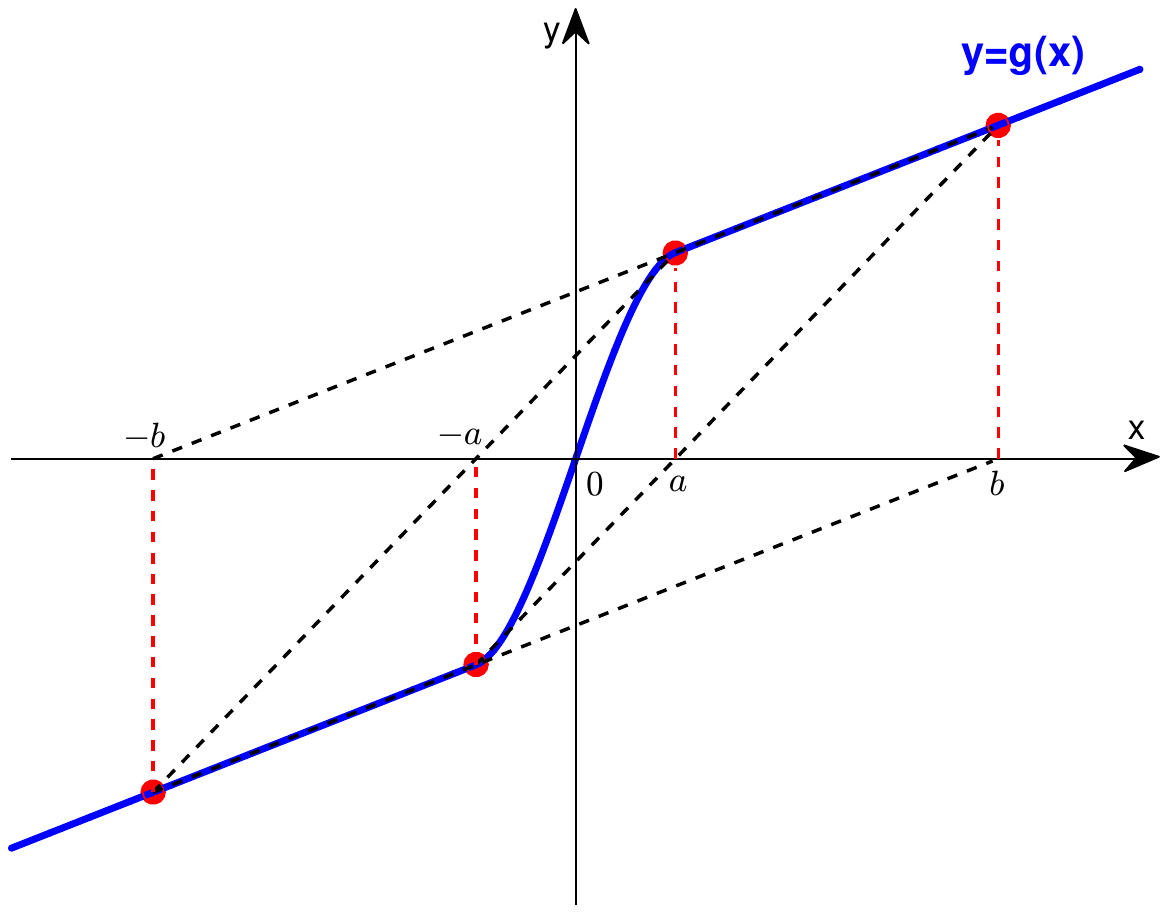}
	\vspace{-9.3cm}
	\caption{Cyclic iterates generated by the BB method for function \eqref{f}.}\label{fig_g}
\end{figure}
It can be easily verified that the function $f(x)$ is twice continuously differentiable with
$$
1/2 \le f''(x) \le c_1, \quad \forall x \in R^1.
$$
This means that this function is strongly convex, and its first derivative is Lipschitz-continuous.

For any univariate objective function, there is no difference between BB1 and BB2 versions, and they are equivalent to the secant method applied to the first derivative. For function \eqref{f}, if to initiate the BB method with $x_0 = -b$ and $x_1 = -a$, then the subsequent iterates are
\begin{eqnarray*}
&&x_2 = x_1 - \frac{x_1-x_0}{g(x_1)-g(x_0)}g(x_1)= b,\\
&&x_3 = x_2 - \frac{x_2-x_1}{g(x_2)-g(x_1)}g(x_2)= a,\\
&&x_4 = x_3 - \frac{x_3-x_2}{g(x_3)-g(x_2)}g(x_3)= -b = x_0,\\
&&x_5 = x_4 - \frac{x_4-x_3}{g(x_4)-g(x_3)}g(x_4)= -a = x_1.\\
\end{eqnarray*}
This clearly shows that the BB method cycles between four points (see Figure \ref{fig_g}). The presented counter-example can be easily extended to $n$-dimensional case.
As an example, one can consider a separable objective function equal to the sum of any number of functions of the form \eqref{f}, where no variable appears in more than one of these functions.

After motivating the necessity of stabilizing the BB method, we can now proceed to presenting the basic idea of our stabilized BB algorithm, where $\Delta > 0$ is a parameter. It consists in choosing the stepsize in \eqref{eqitr} in the way that $\|x_{k+1}-x_k\| = \Delta$, whenever $\|\alpha_k^{BB} g_k\| > \Delta$,
i.e. $\alpha_k^{BB} > \Delta / \|g_k\|$. In other cases, we choose $\alpha_k =  \alpha_k^{BB}$, which results in $\|x_{k+1}-x_k\| \le \Delta$.
Thus, denoting
$$\alpha_k^{stab} = \frac{\Delta }{\|g_k\|},$$
we propose to choose
\begin{equation}\label{alphaBB}
\alpha_k = \min \{\alpha_k^{BB}, \, \alpha_k^{stab}\}.
\end{equation}
Here $\alpha_k^{BB} = \alpha_k^{BB1}$ or $\alpha_k^{BB} = \alpha_k^{BB2}$, depending on the specific BB method in \eqref{bb}. A formal description of our stabilized BB algorithm follows.

\begin{algorithm}\label{BBstab}
BBstab.
\rm{\begin{tabbing}
	\rule{\textwidth}{1px}\\
	\hspace{0.5cm}\=\hspace{0.5cm}\=\hspace{0.5cm}\=\hspace{0.5cm}\=\kill
	\textbf{Given:} initial points $x_0, x_1 \in \mathbb{R}^{n}$ such that $x_0 \neq x_1$, and scalar $\Delta > 0$.\\
	\rule{\textwidth}{0.5px}\\
	Evaluate $g_0$ and $g_1$. \\
	\textbf{for} $k=1,2, \ldots$ \textbf{do}\\
	\> \textbf{if} $g_k=0$ \textbf{then} stop. \\
	\> Set $s_{k-1} \leftarrow x_k - x_{k-1}$ and $y_{k-1} \leftarrow g_k - g_{k-1}$. \\
	\> Compute $\alpha_k$ by formula \eqref{alphaBB}. \\
	\> Set $x_{k+1} \leftarrow x_k - \alpha_k g_k$ and evaluate $g_{k+1}$. \\
	\textbf{end (for)}
	\\
	\rule{\textwidth}{0.5px}
\end{tabbing}}
\end{algorithm}

This algorithm will be refereed to as BB1stab or BB2stab depending on the corresponding choice of $\alpha_k^{BB}$ in \eqref{bb}. Note that, for $\Delta = +\infty$, it reduces to the underlying standard BB algorithm.

\section{Convergence Analysis}
\label{sec:convergence}
In this section, global convergence of the BBstab algorithm will be proved. Whenever iterates $\{x_k\}$ are mentioned, they are assumed to be generated by BBstab, where it is required that
$x_0 \neq x_1$.

Throughout this section, the objective function is assumed to comply with the following requirement.
\begin{description}
	\item[A1.]
The function $f: R^n \rightarrow R^1$ is twice continuously differentiable, and there exist positive constants $\Lambda_1 \le \Lambda_2$ such that
\begin{equation}\label{convexity}
\Lambda_1 \|v\|^2 \le v^T \nabla^2f(x) v \le \Lambda_2 \|v\|^2, \quad \forall x, v \in R^n.
\end{equation}
\end{description}

This assumption implies that
\begin{equation}\label{g_x}
\Lambda_1 \|x - x^*\| \le \|g(x)\| \le \Lambda_2 \|x - x^*\|,
\quad \forall x \in R^n.
\end{equation}
Extra assumptions are introduced below in proper places.

We shall use the following notation:
\begin{equation*}
\begin{split}
\Omega_1 & = \{x \in R^n: \ \|g(x)\| \le \Lambda_1 \Delta\},
\\
\Omega_2 & = \{x \in R^n: \ \Lambda_1 \Delta\ < \|g(x)\| \le \Lambda_2 \Delta\},
\\
\Omega_3 & = \{x \in R^n: \ \Lambda_2 \Delta\ < \|g(x)\|\},
\\
\Omega_{3'} & = \{x \in R^n: \ \Lambda_2 \Delta\ < \|g(x)\| \le \varkappa \Lambda_2 \Delta\},
\end{split}
\end{equation*}
which will be motivated later. Here
$$
\varkappa  = \dfrac{\Lambda_2}{\Lambda_1}.
$$
Obviously, $\Omega_{3'} \subset \Omega_3$, and
$\Omega_{1,2,3} = R^n$, where
$\Omega_{1,2,3} = \Omega_1 \cup \Omega_2 \cup \Omega_3$.
We shall use similar notation for other unions of sets $\Omega_i$.

Inequalities \eqref{convexity} ensure that
\begin{equation}\label{alphaBBbound}
\dfrac{1}{\Lambda_2} \le \alpha_k^{BB} \le \dfrac{1}{\Lambda_1}, \quad \forall k \ge 1,
\end{equation}
which in turn means that
\begin{equation}\label{alpha_k_bound}
\alpha_k \le \min \left\lbrace
\frac{\Delta}{\|g_k\|}, \frac{1}{\Lambda_1}\right\rbrace ,
\quad \forall k \ge 1,
\end{equation}
and
\begin{equation}\label{alpha_large_k}
\frac{1}{\varkappa \Lambda_2} \le \alpha_k \le \frac{1}{\Lambda_1},
\quad \forall x_k \in \Omega_{1,2,3'}.
\end{equation}
These bounds justify the implications
\begin{equation}\label{x_k_alpha_k}
\begin{split}
x_k \in \Omega_1 \quad & \Rightarrow \quad \alpha_k = \alpha_k^{BB},
\\
x_k \in \Omega_2 \quad & \Rightarrow \quad \alpha_k =
\min\{\alpha_k^{BB}, \, \alpha_k^{stab}\},
\\
x_k \in \Omega_3 \quad & \Rightarrow \quad \alpha_k =
\alpha_k^{stab}.
\end{split}
\end{equation}

We can now prove the following result.
\begin{lemma}\label{lem:g_k}
Let $x_0, x_1 \in R^n$ be arbitrary starting points. Then for any $\Delta > 0$, the iterates $\{x_k\}$ have the property that
\begin{equation}\label{norm_g_k}
\|g_{k+1}\| \le \left\lbrace
\begin{array}{rl}
q_k \|g_k\|, & {\rm if}\  x_k \in \Omega_3, \\
\varkappa \|g_k\|, & {\rm otherwise},
\end{array}
\right.
\quad \forall k \ge 1,
\end{equation}
where
$$
q_k = 1 - \frac{\Lambda_1 \Delta}{\|g_k\|}.
$$
\end{lemma}

\begin{proof}
Using Assumption A1, we get
$$
g_{k+1} = g_k - \alpha_k H_k g_k,
$$
where the matrix $H_k = \int_0^1 \nabla^2 f(x_k + t s_k) dt$ is symmetric, and it fulfills the relations
$$
\Lambda_1 I \preceq H_k \preceq \Lambda_2 I.
$$
Clearly,
\begin{equation}\label{norm_g_k_2}
\|g_{k+1}\| \le \|I - \alpha_k H_k\| \|g_k\|.
\end{equation}

Consider, first, the case when $x_k \in \Omega_3$. Using the inequality $\Lambda_2 \Delta < \|g(x)\|$ and relations \eqref{x_k_alpha_k}, we can derive for \eqref{norm_g_k_2} the following upper bound
$$
\|I - \alpha_k^{stab} H_k\|
= \max_{\|v\|=1} \left| 1 - \alpha_k^{stab} v^T H_k v\right|
= 1 - \alpha_k^{stab} \min_{\|v\|=1} v^T H_k v
\le 1 - \frac{\Lambda_1 \Delta}{\|g_k\|}.
$$
This proves the upper inequality in \eqref{norm_g_k}.

Suppose now that $x_k \in \Omega_{1,2}$, i.e., $\|g_k\| \le \Lambda_2 \Delta$. Then, using \eqref{alpha_k_bound}, we get the bounds
$\Lambda_2^{-1} \le \alpha_k \le \Lambda_1^{-1}$, which together with the inequalities
$\Lambda_1 \le \|H_k\| \le \Lambda_2$ yield
$$
\|I - \alpha_k H_k\| \le \max \{1 - \varkappa^{-1}, \varkappa - 1\} = \varkappa -1 < \varkappa.
$$
By combining this estimate with \eqref{norm_g_k_2}, we finally prove the lower inequality in \eqref{norm_g_k}.
\hfill $\Box$ \end{proof}

Lemma~\ref{lem:g_k} implies that the stabilization steps have the following properties
\begin{equation}\label{q_in_01}
q_k \in (0,1), \quad \forall x_k \in \Omega_3,
\end{equation}
\begin{equation}\label{q_decrease}
q_{k+1} < q_k, \quad \forall x_k, x_{k+1}\in \Omega_3.
\end{equation}

Next, we prove that, after a finite number of iterations, all iterates belong to the bounded set $\Omega_{1,2,3'}$.
\begin{lemma}\label{lem:finite_step}
For any $x_0, x_1 \in R^n$ and $\Delta > 0$, there exists an integer $K \ge 1$ such that the inequality
\begin{equation}\label{bounded_g}
\|g_k\| \le \varkappa \Lambda_2 \Delta
\end{equation}
holds, that is $x_k \in \Omega_{1,2,3'}$, for all $k \ge K$. Moreover, $K$ is the iteration number corresponding to the first iterate $x_K$ that belongs to $\Omega_{1,2,3'}$.
\end{lemma}
\begin{proof}
Notice that \eqref{bounded_g} is satisfied if and only if $x_k \in \Omega_{1,2,3'}$.
We first show that if $x_k \in \Omega_{1,2,3'}$, then so does the next iterate. Indeed, in view of \eqref{norm_g_k} and \eqref{q_in_01}, if $x_k \in \Omega_{3'}$, then $x_{k+1} \in \Omega_{1,2,3'}$. On the other hand, if $x_k \in \Omega_{1,2}$, i.e. $\|g_k\| \le \Lambda_2 \Delta$, then, by Lemma~\ref{lem:g_k}, we have $\|g_{k+1}\| \le \varkappa \Lambda_2 \Delta$.
	
Suppose now that $x_1 \in \Omega_{3} \setminus \Omega_{3'}$. Then it immediately follows from relations \eqref{q_in_01} and \eqref{q_decrease}, that there exists $K > 1$ such that $x_K \in \Omega_{1,2,3'}$. As it was shown above, this means that $x_k \in \Omega_{1,2,3'}$ for all $k \ge K$.
\hfill $\Box$ \end{proof}

It follows from \eqref{q_in_01} that, when iterates belong to the set $\Omega_3$, the value $\|g_k\|$ monotonically decreases as indicated by \eqref{norm_g_k}. Furthermore, the actual decrease may speed-up in accordance with \eqref{q_decrease}. When the iterates reach $\Omega_{1,2}$, the decrease is naturally expected to slow down, and this is followed by a non-monotonic behavior of $\|g_k\|$, which is a typical feature of the BB steps.
\begin{figure}[t!]
	\centering
	\vspace{-6.7cm}
	\includegraphics[width=0.99\textwidth]{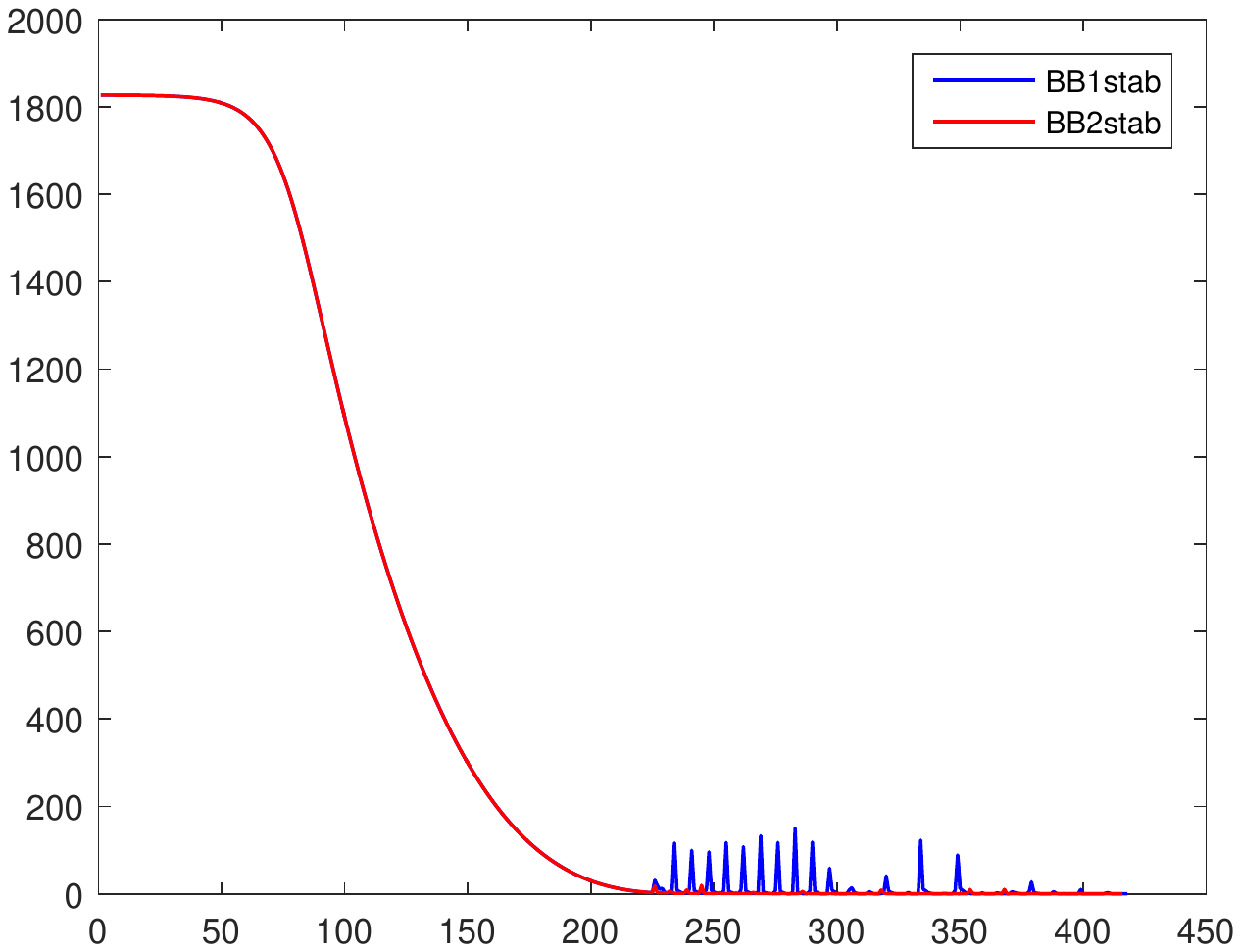}
	\vspace{-6.3cm}
	\caption{Graphs of $\|g_k\|$ for BB1stab and BB2stab with $\Delta = 2$ for Raydan function \eqref{raydan}.}\label{fig:Raydan_stab_g}
\end{figure}
One can observe all these stages in the behavior of BBstab in Figure \ref{fig:Raydan_stab_g}. It presents changes of $\|g_k\|$ with $k$ in the process of minimizing Raydan function \eqref{raydan}. Details of these runs are discussed in Section \ref{sec:numres}. Note that both BB1 and BB2 fail to solve this problem starting from the same points. The figure illustrates the role of stabilization in providing convergence of BBstab. One can clearly recognize the first stage of the process when the stabilization steps ensure a monotonic decrease of $\|g_k\|$. For the BB1stab and BB2stab, the iteration when the standard BB step was used for the first time is 228 and 226, respectively. For them, the last stabilization step was used in iteration 379 and 353, respectively. Observe that the spikes of $\|g_k\|$ produced by BB1 is much larger than those for BB2.

Lemma~\ref{lem:finite_step} allows us to deduce
an interesting property of the BB method, namely, that if it
generates bounded steps, it cannot generate unbounded iterates because one can choose a sufficiently large $\Delta$, which is not binding. The same lemma
indicates that a proper choice of $\Delta $ allows for BBstab to reach any neighborhood of $x^*$. We use the notation
$$
B_{\delta }(x^*) = \{x \in R^n: \ \|x - x^*\| \le \delta \}.
$$
in the following formulation of this useful feature of BBstab.

\begin{lemma}\label{lem:stab_convergence}
Let $x_0, x_1 \in R^n$ be any starting points. Then for any $\delta > 0$ and positive $\Delta \le  \frac{\delta}{\varkappa^2}$, there exists $K(\Delta) \ge 1$ such that the iterates $\{x_k\}$ satisfy the condition
$$
x_k \in B_{\delta }(x^*), \quad \forall k \ge K(\Delta).
$$
\end{lemma}
\begin{proof}
Combining \eqref{g_x} and Lemma~\ref{lem:finite_step}, we get the relations
$$
\|x_k - x^*\| \le \frac{\|g_k\|}{\Lambda_1}
\le \varkappa^2 \Delta \le \delta,
$$
which are satisfied for all sufficiently large $k$. This completes the proof.
\hfill $\Box$ \end{proof}

We shall make use of Lemma~\ref{lem:finite_step} for proving global convergence result for BBstab. We show also that its local rate of convergence is R-linear, which means that there exist positive $\gamma$ and $c \in (0,1)$ such that
\begin{equation}\label{R-linear}
\|x_{k+1} - x^*\| \le \gamma c^k \|x_{1} - x^*\|.
\end{equation}
These convergence results are based on our convergence analysis presented in the next sub-section for convex quadratic functions.

\subsection{Convergence in Quadratic Case}
\label{sec:quad}
In this sub-section, we focus on minimizing convex quadratic functions of the form
\begin{equation}\label{quadratic}
f(x) = \dfrac{1}{2}x^TAx - b^Tx,
\end{equation}
where the matrix $A \in R^{n \times n}$ is positive definite, and $b \in R^n$.
For these functions, we derive the convergence with R-linear rate. To this end, we
will make use of the following property which is the same as Property A in \cite{dai2003}.

\begin{definition}
We say that the choice of the stepsize in \eqref{eqitr} has property \textbf{P} if there exist an integer $m$ and positive constants $M_1$ and $M_2$ such that, for all $k \ge 1$, \\
(i) $\Lambda_1\le\alpha_{k}^{-1}\le M_1$;\\
(ii) for any integer $\ell\in[1,\,n-1]$ and real number $\epsilon>0$, if $R(k-j,\,\ell)\le \epsilon$ and $(g_{k-j}^{(\ell+1)})^2\ge M_2\epsilon$ hold for $j\in[0,\,\min\{k,\,m\}-1]$, then $\alpha_{k}^{-1}\ge\frac{2}{3}\lambda_{\ell+1}$.
\end{definition}

\begin{theorem}\label{th:conv_quad}
Let $x_0, x_1 \in R^n$ be arbitrary starting points. Then for any $\Delta > 0$, the sequence $\{x_k\}$ converges to $x^*$ with {\rm R}-linear rate. Moreover, there exists a positive integer $\bar{j}$, such that, for any $\Delta > 0$,
$x_0 \in R^n$
and $x_1 \in \Omega_{1,2,3'}$, the inequality
$$
\|g_{k+\bar{j}}\| \le \frac{1}{2}\|g_k\|
$$
holds for all $k \ge 1$.
\end{theorem}

\begin{proof} It is well known that the BB method is invariant under orthogonal transformation of the variables and, as it can be easily seen, so does its stabilized version.
Hence, we can assume without loss of generality that the matrix $A$ is of the form
\begin{equation}\label{eigen}
A = {\rm diag}(\lambda_1,\lambda_2,\ldots,\lambda_n),
\end{equation}
where $0 < \Lambda_1 = \lambda_1 < \lambda_2 < \ldots < \lambda_n = \Lambda_2$. Here, like it is often done for the gradient methods (see, e.g., \cite{raydan1993Barzilai}), it is assumed without loss of generality that the matrix $A$ has distinct eigenvalues. Then denoting the $i$-th component of $g_k$ by $g_k^{(i)}$, we have
\begin{equation}\label{i-th_g_k}
g_{k+1}^{(i)} = (1 - \alpha_k \lambda_i)g_k^{(i)},
\quad  i = 1, 2, \ldots , n.
\end{equation}
We will also make use of the following
notation:
$$
R(k,\,\ell)=\sum\limits_{i=1}^{\ell}(g_k^{(i)})^2.
$$

Firstly, we prove that the step size $\alpha_{k}$ has property \textbf{P}.  Lemma~\ref{lem:finite_step} ensures that $x_k \in \Omega_{1,2,3'}$ for all $k \ge 1$. Then the bounds
\eqref{alpha_large_k} show that $\alpha_k$ has property \textbf{P}(i) with $M_1=\varkappa \Lambda_2$.

Next, we will show, for any integer $\ell\in[1,\,n-1]$ and real number $\epsilon>0$, that the inequality $\alpha_{k}^{-1}\ge\frac{2}{3}\lambda_{\ell+1}$ is satisfied, whenever $R(k-1,\,\ell)\le \epsilon$ and $(g_{k-1}^{(\ell+1)})^2\ge 2\epsilon$.
This will be done separately for BB1- and BB2-based iterates.

For the BB1 case, we have
\begin{eqnarray*}
\alpha_k^{-1} &\ge& \frac{  g_{k-1}^T Ag_{k-1} }{  \|g_{k-1}\|^2 }
=\frac{ \sum\limits_{i=1}^{n} \lambda_i(g_{k-1}^{(i)})^2 }{  R(k-1,\,n )}
\ge\frac{ \lambda_{\ell+1}\sum\limits_{i=\ell+1}^{n}  (g_{k-1}^{(i)})^2 }{  R(k-1,\,\ell)+\sum\limits_{i=\ell+1}^{n} (g_{k-1}^{(i)})^2}\\
&\ge&\frac{ \lambda_{\ell+1}\sum\limits_{i=\ell+1}^{n}  (g_{k-1}^{(i)})^2 }{  \epsilon+\sum\limits_{i=\ell+1}^{n} (g_{k-1}^{(i)})^2}
\ge\frac{ 2\lambda_{\ell+1}\epsilon }{  \epsilon+2\epsilon}
=\frac{ 2}{  3}\lambda_{\ell+1}.
\end{eqnarray*}

For BB2, we obtain
\begin{eqnarray*}
(\alpha_k)^{-1} &\ge& \frac{ g_{k-1}^TA^2g_{k-1}  }{ g_{k-1}^TAg_{k-1} }
\ge\frac{ \lambda_{\ell+1}\sum\limits_{i=\ell+1}^{n} \lambda_{i} (g_{k-1}^{(i)})^2 }{ \lambda_{\ell+1} R(k-1,\,\ell)+\sum\limits_{i=\ell+1}^{n} \lambda_{i}(g_{k-1}^{(i)})^2}\\
&\ge& \frac{ \lambda_{\ell+1}^2 (g_{k-1}^{(\ell+1)})^2 }{ \lambda_{\ell+1} R(k-1,\,\ell)+ \lambda_{\ell+1}(g_{k-1}^{(\ell+1)})^2}
\ge\frac{ 2\lambda_{\ell+1}\epsilon }{  \epsilon+2\epsilon}
=\frac{ 2}{  3}\lambda_{\ell+1}.
\end{eqnarray*}
Thus, \textbf{P}(ii) holds for $m=2$ and $M_2=2$. This implies that BBstab stepsize $\alpha_k$ satisfies \textbf{P}. Then we can skip the rest of the proof because it is similar to the proof of Theorem 4.1 in \cite{dai2003}.
\hfill$\Box$\end{proof}

It should be emphasized that, in this lemma, the value of $\bar{j}$ depends only on $\Lambda_1$ and $\Lambda_2$.

\subsection{Convergence in General Case}
\label{sec:general}
For nonquadratic functions, we shall first prove local R-linear convergence of BBstab. This result will then be used for showing that it converges from any starting point.

Throughout this sub-section, we need to additionally assume that the Hessian matrix $\nabla^2 f(x)$ is Lipschitz-continuous at $x^*$. In what follows, we use the notation $H = \nabla^2 f(x^*)$.
\begin{description}
	\item[A2.]
	There exist a radius $\rho > 0$ and a Lipschitz constant $L \ge 0$ such that
	$$
	\|\nabla^2 f(x) - H\| \le L \|x - x^*\|, \quad
	\forall x \in B_{\rho} (x^*).
	$$
\end{description}

This assumption implies that
\begin{equation}\label{L}
\|g(x) - H(x - x^*)\| \le \dfrac{L}{2} \|x - x^*\|^2,
\quad \forall x \in B_{\rho} (x^*).
\end{equation}

The second-order Taylor approximation to $f$ around $x^*$ is given by the quadratic function
\begin{equation}\label{nona2}
  \hat{f}(x)=f(x^*)+\frac{1}{2}(x-x^*)^TH(x-x^*).
\end{equation}
Define new iterates $\hat{x}_{k,j}$ associated with $\hat{f}$ as follows:
\begin{equation}\label{nona3}
  \left\{
    \begin{array}{l}
     \hat{x}_{k,0}=x_k,\\
     \hat{x}_{k,j+1}=\hat{x}_{k,j}-\hat{\alpha}_{k,j}\hat{g}_{k,j},\quad j\ge0,
    \end{array}
  \right.
\end{equation}
where
\begin{equation*}
   \hat{\alpha}_{k,j} =
   \left\{
    \begin{array}{ll}
     \alpha_k,&{\rm if}~j=0,\\
     \min\{\hat{\alpha}_{k,j}^{BB},\,\hat{\alpha}_{k,j}^{stab}\},&{\rm otherwise}.
    \end{array}
  \right.
\end{equation*}
Here $\hat{\alpha}_{k,j}^{BB} = \hat{\alpha}_{k,j}^{BB1}$ or
$\hat{\alpha}_{k,j}^{BB} = \hat{\alpha}_{k,j}^{BB2}$ and
$\hat{\alpha}_{k,j}^{stab} = \frac{\Delta}{\|\hat{g}_{k,j}\|}$ with
\begin{eqnarray*}
\hat{\alpha}_{k,j}^{BB1}=\dfrac{ \hat{s}_{k+j-1}^T\hat{s}_{k+j-1} }{ \hat{s}_{k+j-1}^T\hat{y}_{k+j-1}},
&&
\hat{\alpha}_{k,j}^{BB2}=\dfrac{ \hat{s}_{k+j-1}^T\hat{y}_{k+j-1}}{ \hat{y}_{k+j-1}^T\hat{y}_{k+j-1} },
\end{eqnarray*}
$\hat{s}_{k+j-1}=\hat{x}_{k,j}-\hat{x}_{k,j-1}$, $\hat{g}_{k,j}=H(\hat{x}_{k,j}-x^*)$ and $\hat{y}_{k+j-1}=\hat{g}_{k,j}-\hat{g}_{k,j-1}$.
In what follows, whenever we mention $\hat{x}_{k,j}$ and $\hat{\alpha}_{k,j}$, they are assumed to be generated as defined above.

The next result follows immediately from Theorem~\ref{th:conv_quad}.

\begin{lemma}\label{lem:linear_R_rate}
Let $\Delta > 0$ be any scalar, such that $\Omega_{1,2,3'} \subseteq B_{\rho}(x^*)$.
Then there exists a positive integer $\bar{j}$, dependent only on $\Lambda_1$ and $\Lambda_2$, such that, for any $x_{k-1} \in R^n$ and $x_k\in\Omega_{1,2,3'}$, the inequality holds
$$\| \hat{g}_{k,\bar{j}} \| \le \dfrac{1}{2} \| \hat{g}_{k,0} \|.$$
\end{lemma}

It can be easily seen that if $x_k \in \Omega_{1,2,3'}$, then all corresponding $\hat{x}_{k,j} \in \Omega_{1,2,3'}$. In this case, BBstab stepsize $\alpha_k$ satisfies the bounds \eqref{alpha_large_k}, and similarly for
$\hat{\alpha}_{k,j}$, we have the bounds
\begin{equation}\label{hat_alpha_large_k}
\frac{1}{\varkappa \Lambda_2} \le \hat{\alpha}_{k,j} \le \frac{1}{\Lambda_1},
\quad \forall j \ge 0.
\end{equation}

The following result will be used for proving local R-linear convergence.
\begin{lemma}\label{lem:relation}
Let integer $\bar{j} \ge 1$ be arbitrary.
Then there exist positive scalars $\bar{\Delta}$ and $\gamma$ with the following property:
for any $\Delta\in(\,0,\, \bar{\Delta}\,]$, $x_{k-1} \in R^n$, $x_k\in \Omega_{1,2,3'}\subset B_{\rho}(x^*)$ and
$m\in[0,~\bar{j}]$, for which
\begin{equation}\label{nona9}
  \|\hat{g}_{k,j}\| \ge\frac{1}{2}\|\hat{g}_{k,0}\|, \quad\forall j\in[0,\max\{0, m-1\}],
\end{equation}
we have the inequality
\begin{equation}\label{nona10}
  \|x_{k+j}-\hat{x}_{k,j}\| \le \gamma\|x_k-x^*\| ^2
\end{equation}
satisfied for all $j\in[0, m]$.
\end{lemma}

\begin{proof}
Throughout the proof, let $c$ denote a generic positive constant, which may depend on some of fixed constants, such as $\bar{\Delta}$, $\bar{j}$, $\Lambda_1$, $\Lambda_2$ or $L$, but not on
the choice of $\Delta$ or $x_k\in \Omega_{1,2,3'}\subset B_{\rho}(x^*)$.
For brevity, we will use the same notation in all inequalities, even though every specific value of $c$ depends on the one, where it is used. What is important is that the number of these inequalities is finite.

We first notice that, by Lemma~\ref{lem:finite_step}, the relation $x_{k+j} \in \Omega_{1,2,3'} \subseteq B_{\rho}(x^*)$ holds for all $j \ge 0$.
The process of proving \eqref{nona10} will be combined with showing that the inequalities
\begin{equation}\label{nona11}
  \|g(x_{k+j})-\hat{g}(\hat{x}_{k,j})\| \le c\|x_k-x^*\| ^2,
\end{equation}
\begin{equation}\label{nona12}
  \|s_{k+j}\| \le c\|x_k-x^*\| ,
\end{equation}
\begin{equation}\label{nona13}
  |\alpha_{k+j}-\hat{\alpha}_{k,j}|\le c\|x_k-x^*\| ,
\end{equation}
are satisfied for all $j\in[0,\, m]$.

The proof of \eqref{nona10}-\eqref{nona13} is by induction on $ m$. For $ m=0$, noticing that $\hat{x}_{k,0}=x_k$, $\hat{\alpha}_{k,0}=\alpha_k$ and $s_{k}=-\alpha_k g_k$, by \eqref{g_x}, \eqref{alpha_large_k} and \eqref{L}, we can immediately get \eqref{nona10}-\eqref{nona13} satisfied for $j=0$.

Suppose that there exist $M\in[1,\,\bar{j})$ and $\bar{\Delta}>0$ with the property that if \eqref{nona9} holds for any $ m\in[0,\, M-1]$, then \eqref{nona10}-\eqref{nona13} are satisfied for all $j\in[0,\, m]$. Next, we shall show that for a smaller choice of $\bar{\Delta}>0$, we can replace $M$ by $M+1$. Hence, we suppose that \eqref{nona9} holds for all $j\in[0,\,M]$. Since \eqref{nona9} holds for all $j\in[0,\,M-1]$, it follows from the induction hypothesis and \eqref{nona12} that
\begin{equation}\label{nona43}
  \|x_{k+M+1}-x^* \|  \le \|x_{k}-x^* \| +\sum\limits_{i=0}^{M}\|s_{k+i}\|
  \le c\|x_{k}-x^* \| .
\end{equation}
By analogy with the proof of Lemma 2.2 in \cite{dai2006cyclic}, we derive from \eqref{g_x}, \eqref{alpha_large_k}, \eqref{L}, \eqref{hat_alpha_large_k}, \eqref{nona43} and the induction hypothesis that \eqref{nona10}-\eqref{nona12} hold for $j=M+1$. Then we just need to show that
\begin{equation}\label{nona44}
  |\alpha_{k+M+1}-\hat{\alpha}_{k,M+1}|\le c\|x_k-x^*\| .
\end{equation}
It follows from \eqref{g_x} that
$$\| x_k - x^* \| \le \frac{\|g_k\|}{\Lambda_1} \le \varkappa^2 \Delta \le \varkappa^2 \bar{\Delta}. $$
Then by choosing any $\bar{\Delta} < 1 / (2 \gamma \varkappa^3)$, using relations \eqref{convexity}, \eqref{alpha_large_k},  \eqref{hat_alpha_large_k}-\eqref{nona10}, \eqref{nona12} and the same reasoning as in the proof of Lemma 2.2 in \cite{dai2006cyclic}, we obtain
\begin{equation}\label{BBstepbound}
  | \alpha_{k+M+1}^{BB} - \hat{\alpha}_{k, M+1}^{BB} |
  \le c\|x_k-x^*\| .
\end{equation}
In the following, the proof of \eqref{nona44} will be done by separately considering four different cases.

Case I: $ \alpha_{k+M+1}^{BB} \le \alpha_{k+M+1}^{stab} $ and
$ \hat{\alpha}_{k, M+1}^{BB}  \le \hat{\alpha}_{k, M+1}^{stab} $.\\
Then \eqref{BBstepbound} directly leads to
$$
  |\alpha_{k+M+1}-\hat{\alpha}_{k,M+1}| = | \alpha_{k+M+1}^{BB} - \hat{\alpha}_{k, M+1}^{BB} |
  \le c\|x_k-x^*\| .
$$

Case II: $ \alpha_{k+M+1}^{BB} \le \alpha_{k+M+1}^{stab} $ and
$ \hat{\alpha}_{k, M+1}^{BB}  > \hat{\alpha}_{k, M+1}^{stab} $.\\
If $\hat{\alpha}_{k, M+1}^{stab} \ge \alpha_{k+M+1}^{BB}$, then \eqref{BBstepbound} implies
$$
|\alpha_{k+M+1}-\hat{\alpha}_{k,M+1}|
= \hat{\alpha}_{k, M+1}^{stab}- \alpha_{k+M+1}^{BB}
<  \hat{\alpha}_{k, M+1}^{BB} -  \alpha_{k+M+1}^{BB} \le c\|x_k-x^*\|.
$$
Suppose now that $\hat{\alpha}_{k, M+1}^{stab} < \alpha_{k+M+1}^{BB}$. Then we have
\begin{equation}\label{stepsizeinequality}
  |\alpha_{k+M+1}-\hat{\alpha}_{k,M+1}|
=  \alpha_{k+M+1}^{BB} - \hat{\alpha}_{k, M+1}^{stab}
\le  \alpha_{k+M+1}^{stab} - \hat{\alpha}_{k, M+1}^{stab}.
\end{equation}
It follows from \eqref{alpha_large_k} and \eqref{hat_alpha_large_k} that
\begin{equation}\label{nona41}
 \|\hat{g}_{k,M+1}\| = \frac{\Delta}{\hat{\alpha}_{k, M+1}^{stab}}
 > \frac{\Delta}{\hat{\alpha}_{k, M+1}^{BB}}
 \ge \Delta\Lambda_1.
\end{equation}
By \eqref{g_x} and \eqref{L}, we get
$$\|g_{k+M+1} - \hat{g}_{k,M+1}\|
\le \dfrac{L}{2} \|x_{k+M+1} - x^*\|^2
\le \dfrac{L}{2\Lambda_1^2} \|g_{k+M+1}\|^2
\le \dfrac{1}{2}\varkappa^4 \Delta^2 L .$$
This along with \eqref{nona41} leads to
$$
\|g_{k+M+1}\| \ge \|\hat{g}_{k,M+1}\| - \|g_{k+M+1} - \hat{g}_{k,M+1}\|
\ge \Delta\Big(\Lambda_1 - \dfrac{1}{2}\varkappa^4 \Delta  L \Big)
\ge \Delta C(\bar{\Delta}),
$$
where
$C(\bar{\Delta}) = \Lambda_1 - \varkappa^4 \bar{\Delta} L / 2 > 0$ whenever $\bar{\Delta} < 2\Lambda_1 / (\varkappa^4  L)$. Then we obtain
\begin{align*}
|\alpha_{k+M+1}^{stab}-\hat{\alpha}_{k,M+1}^{stab}| &=
\left| \frac{\Delta}{\|g_{k+M+1}\|}-\frac{\Delta}{\|\hat{g}_{k,M+1}\|}\right|
=\Delta \frac{|\|\hat{g}_{k,M+1}\|-\|g_{k+M+1}\||}{\|g_{k+M+1}\|\|\hat{g}_{k,M+1}\|}
\\
&\le\Delta \frac{\|\hat{g}_{k,M+1}-g_{k+M+1}\|}{\|g_{k+M+1}\|\|\hat{g}_{k,M+1}\|}
\le \frac{ c\|x_k-x^*\| ^2 }{\Delta\Lambda_1 C(\bar{\Delta})}
\le  \frac{ c \|g_k\| \|x_k-x^*\|   }{\Delta\Lambda_1^2 C(\bar{\Delta})}
\\
& \le  \frac{ c \varkappa \Lambda_2 \Delta \|x_k-x^*\|   }{\Delta\Lambda_1^2 C(\bar{\Delta})}
=  \frac{ c \varkappa^2 }{\Lambda_1 C(\bar{\Delta})} \|x_k-x^*\|.
\end{align*}
This together with \eqref{stepsizeinequality} shows that \eqref{nona44} holds.

Case III: $ \alpha_{k+M+1}^{BB} > \alpha_{k+M+1}^{stab} $ and
$ \hat{\alpha}_{k, M+1}^{BB}  \le \hat{\alpha}_{k, M+1}^{stab} $. \\
If $\alpha_{k+M+1}^{stab} \ge \hat{\alpha}_{k, M+1}^{BB}$, then by \eqref{BBstepbound}, we have
$$
|\alpha_{k+M+1}-\hat{\alpha}_{k,M+1}|
=   \alpha_{k+M+1}^{stab} - \hat{\alpha}_{k, M+1}^{BB}
\le  \alpha_{k+M+1}^{BB} - \hat{\alpha}_{k, M+1}^{BB} \le c\|x_{k}-x^*\|.
$$
Suppose now that $\alpha_{k+M+1}^{stab} < \hat{\alpha}_{k, M+1}^{BB}$. Then we get
$$
|\alpha_{k+M+1}-\hat{\alpha}_{k,M+1}|
=  \hat{\alpha}_{k, M+1}^{BB} - \alpha_{k+M+1}^{stab}
\le  \hat{\alpha}_{k, M+1}^{stab} - \alpha_{k+M+1}^{stab}.
$$
To use the same reasoning as in Case II, we need to have lower bounds for $\|g_{k+M+1}\|$ and
$\|\hat{g}_{k,M+1}\|$. To this end,
applying \eqref{alpha_large_k} and \eqref{hat_alpha_large_k}, we obtain
\begin{equation}\label{nona42}
  \|g_{k+M+1}\| = \frac{ \Delta }{ \alpha_{k+M+1}^{stab} }
   > \frac{ \Delta }{\alpha_{k+M+1}^{BB} }
   \ge \Delta\Lambda_1.
\end{equation}
Furthermore, \eqref{g_x}, \eqref{L} and \eqref{nona42} yield
$$
\|\hat{g}_{k,M+1}\| \ge \|g_{k+M+1}\| - \|g_{k+M+1} - \hat{g}_{k,M+1} \|
\ge \Delta C(\bar{\Delta}).
$$
This lower bound is positive whenever $\bar{\Delta} < 2\Lambda_1 / (\varkappa^4  L)$.
The two lower bounds allows us to conclude, by analogy with Case II, that \eqref{nona44} holds.

Case IV: $ \alpha_{k+M+1}^{BB} > \alpha_{k+M+1}^{stab} $ and
$ \hat{\alpha}_{k, M+1}^{BB}  > \hat{\alpha}_{k, M+1}^{stab} $.\\
It follows from \eqref{g_x}, \eqref{nona11}, \eqref{nona41} and \eqref{nona42} that
\begin{align*}
|\alpha_{k+M+1}-\hat{\alpha}_{k,M+1}|
&\le \Delta \frac{\|\hat{g}_{k,M+1}-g_{k+M+1}\|}{\|g_{k+M+1}\|\|\hat{g}_{k,M+1}\|}
\le \Delta \frac{ c\|x_k-x^*\| ^2 }{\Delta^2 \Lambda_1^2}
\le \frac{ c \|g_k\| \|x_k-x^*\|   }{\Delta \Lambda_1^3}
\\
&\le  \frac{ c \varkappa \Lambda_2 \Delta \|x_k-x^*\|   }{\Delta \Lambda_1^3}
\le \frac{ c \varkappa \Lambda_2  }{ \Lambda_1^3} \|x_k-x^*\| .
\end{align*}

Collecting the results in the considered four cases, one can see that \eqref{nona44} is satisfied for any
$$\Delta<\min\left\lbrace \frac{1}{2 \gamma \varkappa^3}, \frac{2\Lambda_1 }{\varkappa^4  L}\right\rbrace .$$
This completes the induction and finally proves that inequalities
\eqref{nona10}-\eqref{nona13}
hold for all $j\in[0,\, m]$.
\hfill $\Box$ \end{proof}

Next we will establish the local convergence property of BBstab for nonquadratic functions.

\begin{theorem}\label{th:local}
There exists positive $\bar{\Delta}$ such that, for any positive $\Delta \le \bar{\Delta}$ and any starting points $x_0,~x_1\in \Omega_{1,2,3'}$, the sequence $\{x_k\}$ converges to $x^*$ with {\rm R}-linear rate.
\end{theorem}

Lemma~\ref{lem:relation} allows us to skip the proof of this theorem because the reasoning is similar to the proof of Theorem~2.3 in \cite{dai2006cyclic}.

We complete the analysis by presenting the following global convergence result.

\begin{theorem}\label{th:global}
There exists positive $\bar{\Delta}$ such that, for any positive $\Delta \le \bar{\Delta}$ and any starting points $x_0,~x_1\in R^n$, the sequence $\{x_k\}$ converges to $x^*$ with {\rm R}-linear rate.
\end{theorem}
\begin{proof}
Let $\bar{\Delta} > 0$ be given by Theorem~\ref{th:local}, which ensures local convergence to $x^*$.
According to Lemma~\ref{lem:finite_step}, after a finite number of BBstab iterations, all iterates will belong to $\Omega_{1,2,3'}$. This finally proves global convergence with R-linear rate.
\hfill $\Box$ \end{proof}

\section{Numerical Results}
\label{sec:numres}
Our algorithms were implemented in MATLAB.
The algorithms are terminated when either the number of iterations exceeds $10^5$, or
$$
\|g_k\| \le 10^{-6}\cdot \|g_0\|.
$$
In the next two subsections, results of numerical experiments are presented separately for quadratic and nonquadratic test functions.

A successful value of $\Delta$ is obviously problem dependent. In our implementation, we try to estimate its order of magnitude by setting $\Delta = +\infty$ for the first few iterations and making use of $\|s_k\|$ produced at these iterations by the standard BB algorithm. At the subsequent iterations, the constant value
\begin{equation}\label{delta}
\Delta = c\cdot \min \{\|s_1\|, \|s_2\|, \|s_3\| \},
\end{equation}
is applied, where $c>0$ is a parameter. It turns out that this adaptive choice of $\Delta$ is less problem dependent.

It is necessary to emphasize that the stabilization was designed not to speed-up the BB method when it safely converges. In such cases, it may increase the number of iterations, which is a negative outcome. The main purpose of the stabilization is to prevent the BB method from making too long steps. This serves for decreasing the number of BB iterations in case of its poor convergence or even making the method convergent when it fails, which is a positive outcome. Outcomes of all these aforementioned types were observed in our numerical experiments with stabilizing the BB method. One can easily recognize them in the tables presented below.

We focus here on demonstrating the potentials of improving convergence for the BB method.
Therefore, our stabilized version is not checked here against another optimization algorithms.
Since the computational cost of one iteration for the BB algorithms are practically the same as for their stabilized versions, only the number of iterations are compared. Notice that the number of iterations is the same as the number of gradient evaluations.

In our numerical experiments,
the BB1 algorithm was generating too long steps more frequently than the BB2 algorithm. This is often caused by relatively too small values of the scalar product $s_{k-1}^Ty_{k-1}$ in the denominator of $\alpha_k^{BB1}$. This explains why the stabilization is, in general,
more important for the BB1 stepsize choice than for the BB2.
Therefore, the numerical results presented here refer mainly to the BB1.

\begin{table}[thb]
	\begin{center}
		\caption{Numerical results for linear systems
			from the SuiteSparse Matrix Collection, Part I.
		}
		\label{tab:linear1}
		\begin{scriptsize}
			\begin{tabular}{|l r r r c|l r r r c|}
				\hline
				\multicolumn{2}{|c}{PROBLEM} & \multicolumn{1}{c}{BB1} & \multicolumn{2}{c|}{BB1stab} & \multicolumn{2}{|c}{PROBLEM} & \multicolumn{1}{c}{BB1} & \multicolumn{2}{c|}{BB1stab}
				\\
				\multicolumn{1}{|c}{name} & \multicolumn{1}{c|}{$n$} & \multicolumn{1}{|c|}{it} & \multicolumn{1}{|c}{it} & \multicolumn{1}{c|}{$c$} & \multicolumn{1}{|c}{name} & \multicolumn{1}{c|}{$n$} & \multicolumn{1}{|c|}{it} & \multicolumn{1}{|c}{it} & \multicolumn{1}{c|}{$c$}
				\\
				\hline
				1138\_bus & 1\,138 & 35\,202 &\bf 21\,384 & 0.3 & ex33 & 1\,733 & 1\,303 & \bf 958 & 0.2
				\\
				2cubes\_sphere & 101\,492 & 5\,576 & \bf 4\,662 & 0.3 & Flan\_1565 & 1\,564\,794 & 13\,781 & 16\,537 & 0.25
				\\
				af\_0\_k101  &  503\,625  &  4\,433  &  \bf 2\,634   &  0.2  &  fv3  &  9\,801  &  449  &  \bf 449   &  0.2
				\\
				af\_1\_k101  &  503\,625  &  2\,473  &  2\,766   &  0.25  &  G2\_circuit  &  150\,102  &  1\,139  & \bf  1\,139   &  0.25
				\\
				af\_2\_k101  &  503\,625  &  4\,034  & \bf 2\,499   &  0.25  &  G3\_circuit  &  1\,585\,478  &  2\,177  &  \bf 2\,177   &  0.2
				\\
				af\_3\_k101  &  503\,625  &  3\,627  & \bf 2\,378   &  0.2  &  Geo\_1438  &  1\,437\,960  &  32\,134  & \bf 29\,095   &  0.3
				\\
				af\_4\_k101  &  503\,625  &  3\,047  &  5\,368   &  0.3  &  gyro  &  17\,361  &  10\,611  &  11\,925   &  0.3
				\\
				af\_5\_k101  &  503\,625  &  2\,397  &  2\,753   &  0.2  &  gyro\_m  &  17\,361  &  3\,325  & \bf 2\,225     &0.25
				\\
				af\_shell3  &  504\,855  &  1\,956  &  4\,565   &  0.3  &  hood  &  220\,542  &  4\,073  &  4\,308   &  0.25
				\\
				af\_shell7  &  504\,855  &  2\,495  &  5\,515   &  0.3  &  Hook\_1498  &  1\,498\,023  &  7\,839  & \bf 7\,358   &  0.25
				\\
				apache1  &  80\,800  &  18\,017  & \bf 9\,143   &  0.2  &  inline\_1  &  503\,712  &  20\,490  & \bf 16\,833   &  0.3
				\\
				apache2  &  715\,176  &  17\,807  & \bf 17\,807   &  0.2  &  jnlbrng1  &  40\,000  &  124  & \bf 108     &0.2
				\\
				audikw\_1  &  943\,695  &  92\,730  & \bf 65\,818   &  0.2  &  Kuu  &  7\,102  &  1\,733  &  \bf 949     &0.3
				\\
				bcsstk08  &  1\,074  &  4\,627  &  5\,113   &  0.3  &  ldoor  &  952\,203  &  9\,133  &  9\,281     &0.3
				\\
				bcsstk09  &  1\,083  &  747  & \bf 713   &  0.3  &  LF10000  &  19\,998  &  48\,867  &  \bf 38\,250     &0.2
				\\
				bcsstk10  &  1\,086  &  3\,416  & \bf 2\,383   &   0.25  &  LFAT5000  &  19\,994  &  22\,358  & \bf 22\,358     &0.25
				\\
				bcsstk11  &  1\,473  &  2\,204  & \bf 1\,699   &  0.2  &  m\_t1  &  97\,578  &  1\,826  &  \bf 1\,826     &0.2
				\\
				bcsstk13  &  2\,003  &  6\,848  &  8\,171   &  0.3  &  mhd3200b  &  3\,200  &  2\,065  &  \bf 2\,065   &  0.2
				\\
				bcsstk14  &  1\,806  &  3\,577  & \bf 2\,682   &  0.25  &  mhd4800b  &  4\,800  &  2\,466  & \bf 2\,466   &  0.2
				\\
				bcsstk15  &  3\,948  &  7\,006  & \bf 4\,872   &  0.25  &  msc01050  &  1\,050  &  15\,187  & \bf 11\,529   &  0.25
				\\
				bcsstk16  &  4\,884  &  401  & \bf 401   &  0.25  &  msc01440  &  1\,440  &  807  & \bf 807     &0.2
				\\
				bcsstk17  &  10\,974  &  27\,014  & \bf 14\,841   &  0.25  &  msc04515  &  4\,515  &  8\,066  & \bf 6\,889   &  0.2
				\\
				bcsstk18  &  11\,948  &  5\,895  & \bf 4\,332   &  0.3  &  msc10848  &  10\,848  &  3\,356  & \bf 3\,356   &  0.2
				\\
				bcsstk21  &  3\,600  &  1\,455  &  1\,594   &  0.25  &  msc23052  &  23\,052  &  19\,088  &  \bf 7\,340   &  0.2
				\\
				bcsstk23  &  3\,134  &  8\,182  & \bf 5\,619   &  0.2  &  msdoor  &  415\,863  &  8\,113  &  \bf 6\,655     &0.25
				\\
				bcsstk24  &  3\,562  &  2\,383  & \bf 1\,537   &  0.3  &  nasa1824  &  1\,824  &  9\,520  &  \bf 6\,515   &  0.3
				\\
				bcsstk25  &  15\,439  &  8\,369  &  8\,971   &  0.25  &  nasa2146  &  2\,146  &  355  &  \bf 355   &  0.2
				\\
				bcsstk26  &  1\,922  &  12\,624  & \bf 8\,761   &  0.2  &  nasa2910  &  2\,910  &  19\,574  & \bf 13\,683   &  0.3
				\\
				bcsstk27  &  1\,224  &  863  &  887   &  0.3  &  nasa4704  &  4\,704  &  43\,448  &  \bf 32\,961   &  0.2
				\\
				bcsstk36  &  23\,052  &  15\,466  & \bf 12\,001   &  0.25  &  nasasrb  &  54\,870  &  10\,302  & \bf 10\,223   &  0.3
				\\
				bcsstk38  &  8\,032  &  1\,584  & \bf 1\,584   &  0.25  &  nd3k  &  9\,000  &  67\,509  &  86\,986 &    0.25
				\\
				bcsstm08  &  1\,074  &  4\,183  & \bf 4\,183   &  0.2  &  nd6k  &  18\,000  &  92\,468  &  \bf 41\,133   &  0.2
				\\
				bcsstm11  &  1\,473  &  623  &  \bf 287   & 0.3  &  nd24k  &  72000  &  84\,165  & \bf 73\,216     &0.3
				\\
				bcsstm12  &  1\,473  &  2\,838  & \bf 2\,375   &  0.3  &  offshore  &  259\,789  &  3\,826  &  3\,949 &    0.3
				\\
				bcsstm23  &  3\,134  &  2\,143  & \bf 1\,857   &  0.25  &  oilpan  &  73\,752  &  4\,647  &  \bf 3\,899   &  0.3
				\\
				bcsstm24  &  3\,562  &  2\,102  & \bf 1\,611   &  0.25  &  olafu  &  16\,146  &  69\,575  &  80\,804     &0.3
				\\
				\hline
			\end{tabular}
		\end{scriptsize}
	\end{center}
\end{table}

\begin{table}[thb]
	\begin{center}
		\caption{Numerical results for linear systems
			from the SuiteSparse Matrix Collection, Part II.
		}
		\label{tab:linear2}
		\begin{scriptsize}
			\begin{tabular}{|l r r r c|l r r r c|}
				\hline
				\multicolumn{2}{|c}{PROBLEM} & \multicolumn{1}{c}{BB1} & \multicolumn{2}{c|}{BB1stab} & \multicolumn{2}{|c}{PROBLEM} & \multicolumn{1}{c}{BB1} & \multicolumn{2}{c|}{BB1stab}
				\\
				\multicolumn{1}{|c}{name} & \multicolumn{1}{c|}{$n$} & \multicolumn{1}{|c|}{it} & \multicolumn{1}{|c}{it} & \multicolumn{1}{c|}{$c$} & \multicolumn{1}{|c}{name} & \multicolumn{1}{c|}{$n$} & \multicolumn{1}{|c|}{it} & \multicolumn{1}{|c}{it} & \multicolumn{1}{c|}{$c$}
				\\
				\hline
				bcsstm25  &  15\,439  &  2\,266  & \bf 2\,119   &  0.2  &  parabolic\_fem  &  525\,825  &  5\,451  & \bf 2\,989 &    0.2
				\\
				bcsstm26  &  1\,922  &  1\,614  &  \bf 1\,239   &  0.2  &  plat1919  &  1\,919  &  3\,297  &  \bf 2\,804   &  0.2
				\\
				bcsstm39  & 46\,772  &  575  &    \bf 575   & 0.2  &  plbuckle  &  1\,282  &  5\,601  &  \bf 3\,726   &  0.3
				\\
				BenElechi1  & 245\,874  &  3\,137  &  \bf 3121  &  0.3  &  Pres\_Poisson  &  14\,822  &  17\,291  & \bf  13\,461   &  0.25
				\\
				bloweybq  & 10\,001  &  107  & \bf  107   &  0.2  &  pwtk  &  21\,7918  &  26\,060  & \bf 21\,798   &  0.25
				\\
				bmw7st\_1  &  141\,347  &  2\,463  &  \bf 2\,463   &  0.2 &  s1rmq4m1  &  5\,489  &  9\,043  & \bf 6\,890 &    0.2
				\\
				bmwcra\_1  &  148\,770  &  86\,966  &  123\,528   &  0.25  &  s1rmt3m1  &  5\,489  &  10\,092  &  11\,576   &  0.25
				\\
				bodyy4  &  17\,546  &  154  &  \bf 154   &  0.25  &  s2rmq4m1  &  5\,489  &  5\,371  &  8\,958   &  0.2
				\\
				bodyy5  &  18\,589  &  405  &  \bf 405   &  0.3  &  s2rmt3m1  &  5\,489  &  7\,850  &  \bf 6\,039   &  0.25
				\\
				bodyy6  &  19\,366  &  809  &  853     &0.3  &  s3dkq4m2  &  90\,449  &  16\,169  &  \bf 16\,169   &  0.2
				\\
				bone010  &  986\,703  &  55\,659  &  \bf 55\,659   &  0.25  &  s3dkt3m2  &  90\,449  &  18\,654  &  \bf 10\,739   &  0.2
				\\
				boneS01  &  127\,224  &  7\,688  &  \bf 5\,669   &  0.2  &  s3rmq4m1  &  5\,489  &  8\,413  &  \bf 7\,848   &  0.25
				\\
				boneS10  &  914\,898  &  28\,584  &  \bf 24\,899   &  0.2  &  s3rmt3m1  &  5\,489  &  16\,901  &  19\,625 &    0.3
				\\
				bundle1  &  10\,581  &  244  &  \bf 244   &  0.2  &  s3rmt3m3  &  5\,357  &  15\,586  &  \bf 6\,737     &0.25
				\\
				cant  &  62\,451  &  19\,609  &  22\,895   &  0.2  &  Serena  &  1\,391\,349  &  47\,765  &  \bf 23\,155   &  0.25
				\\
				cbuckle  &  13\,681  &  6\,963  &  10\,770   &  0.25  &  ship\_001  &  34\,920  &  17\,575  &  \bf 17\,499   &  0.2
				\\
				cfd1  &  70\,656  &  4\,475  &  \bf 3\,555   &  0.2  &  ship\_003  &  121\,728  &  64\,349  &  69\,948   &  0.3
				\\
				cfd2  &  123\,440  &  5\,515  &  8\,145 &    0.25  &  shipsec1  &  140\,874  &  8\,730  &  \bf 6\,681   &  0.2
				\\
				Chem97ZtZ  &  2\,541  &  125  &  \bf 114   &  0.25  &  shipsec5  &  179\,860  &  2\,565  &  3\,113   &  0.3
				\\
				consph  &  83\,334  &  15\,034  &  \bf 11\,232   &  0.25  &  shipsec8  &  114\,919  &  3\,900  &  5\,827   &  0.3
				\\
				crankseg\_1  &  52\,804  &  4\,012  &  \bf 4\,012   &  0.2  &  smt  &  25\,710  &  38\,442  &  \bf 24\,695     &0.25
				\\
				crankseg\_2  &  63\,838  &  4\,914  &  \bf 3\,614   &  0.3  &  sts4098  &  4\,098  &  8\,262  &  12\,042   &  0.2
				\\
				crystm01  &  4\,875  &  100  &  \bf 100   &  0.2  &  t2dah\_e  &  11\,445  &  2\,557  &  \bf 1\,612   &  0.3
				\\
				crystm02  &  13\,965  &  114  &  \bf 114     &0.2  &  t2dal\_e  &  4\,257  &  1\,585  & \bf 1\,171   &  0.25
				\\
				ct20stif  &  52\,329  &  6\,482  &  \bf 6\,482   &  0.25  &  t3dl\_e  &  20\,360  &  503  &  \bf 361   &  0.2
				\\
				cvxbqp1  &  50\,000  &  383  &  \bf 383     &0.2  &  thermal1  &  82\,654  &  5\,812  &  \bf 5\,812   &  0.2
				\\
				Dubcova1  &  16\,129  &  181  &  \bf 181   &  0.2  &  thermal2  &  1\,228\,045  &  22\,201  &  \bf 7\,170   &  0.25
				\\
				Dubcova2  &  65\,025  &  372  &  \bf 348   &  0.3  &  tmt\_sym  &  726\,713  &  40\,335  &  \bf 40\,335   &  0.25
				\\
				Dubcova3  &  146\,689  &  520  &  \bf 429   &  0.2  &  Trefethen\_2000  &  2\,000  &  258  &  \bf 258   &  0.2
				\\
				ex3  &  1\,821  &  508  &  \bf 387   &  0.2  &  Trefethen\_20000  &  20\,000  &  358  &  \bf 358     &0.2
				\\
				ex9  &  3\,363  &  1\,202  &  \bf 1\,202   &  0.3  &  Trefethen\_20000b  &  19\,999  &  404  &  \bf 404     &0.2
				\\
				ex10  &  2\,410  &  3\,038  &  \bf 2\,023   &  0.25  &  vanbody  &  47\,072  &  19\,354  &  \bf 19\,133   &  0.2
				\\
				ex10hs  &  2\,548  &  2\,412  &  \bf 1\,628   &  0.2  &  wathen100  &  30\,401  &  238  &  \bf 238   &  0.25
				\\
				ex13  &  2\,568  &  2\,972  &  \bf 2\,972   &  0.2  &  wathen120  &  36\,441  &  308  &  \bf 308   &  0.2
				\\
				ex15  &  6\,867  &  3\,022  &    3\,298   &  0.3  &  --- & --- & --- & --- & ---\\				
				\hline
			\end{tabular}
		\end{scriptsize}
	\end{center}
\end{table}

\subsection{Quadratic test functions}
A part of the numerical experiments was related to minimizing convex quadratic functions \eqref{quadratic}. This problem is equivalent to solving the system of linear equations
$$
Ax = b.
$$
The matrices in our set of test problems come from the SuiteSparse Matrix Collection \cite{DavisHu-11,sparse}. For generating the vector $b$, we assumed that the solution $x^* = e$, i.e., $b = Ae$, where $e = (1, 1, \ldots , 1)^T$. The total number of problems in our test set is $141$, where the problem size $n$ varies from thousands to millions.

For the adaptive selection of $\Delta$ by formula \eqref{delta}, we tried just a few values of the parameter $c$, namely, $0.2$, $0.25$ and $0.3$. In Tables \ref{tab:linear1} and \ref{tab:linear2}, the number of iterations are reported for algorithms BB1 and BB1stab. For the latter, the best of the three results is presented along with the corresponding value of $c$. If the reported result is the same as for the BB1 algorithm, then it is obvious that the number of iterations remains the same for all values of c larger than the indicated one. The number of iterations, which is not worse than for the BB1 algorithm, are highlighted in this and other tables in this paper. One can see that, comparing with the BB1, its stabilized version is faster in solving $74$ problems, while it is slower in $30$ problems. Furthermore, the reduction in the number of iterations obtained by virtue of the stabilization was often substantial.
We also tested the BB2 and BB2stab algorithms for these same 141 problems. We tried $c\,=\,0.1$, $0.2$, $0.25$ and $0.3$ in the adaptive selection of $\Delta$ by formula \eqref{delta}. Comparing with the BB2, BB2stab is faster in solving $58$ problems, while for the given values of $c$, the stabilization is unable to decrease the number of BB2 iterations in $58$ problems.

\subsection{Nonquadratic test functions}
For general functions,
it is more difficult than for quadratic ones to avoid the cases, when $x_1$ is chosen too close to $x_0$ or too far away of it. In order to avoid such poor choices of these two points, our BBstab algorithms are initialized with only one point, namely, $x_0$. The point $x_1$ is produced in the algorithms by checking if the inequality
$f(x_0 + s_0) < f(x_0)$ is satisfied for $s_0 = - \alpha_0 g_0$, where $\alpha_0 = 1/\|g_0\|_{\infty}$.
Otherwise, a number, typically few, of backtracking steps are performed by dividing the current vector $s_0$ by $4$, while the required inequality is violated.

We begin here by comparing the performance of the BB algorithms and their stabilized versions on the strongly convex Raydan function \eqref{raydan} for $n = 1000$. The point $x_0=-10\cdot e$ was used for starting the algorithms.
The standard BB1 algorithm failed to
solve the problem. After two iterations, an overflow in computing $s_k^Ty_k$
was reported.
If to introduce the bounds $[10^{-30},10^{30}]$ for $\alpha^{BB1}$, like it is often done in practice,
then it also fails, although after a larger number of iterations. Namely, at iteration 123 and all subsequent iterations, an underflow was observed in calculating
$x_{k+1}$ for $\|s_k\| <10^{-26}$.
In these two cases, the standard BB2 also failed.
However, the same test problem for the same $x_0$ was successfully solved by BB1stab and BB2stab with $\Delta = 2$  in $418$ and $416$ iterations, respectively. No bounds, like $[10^{-30},10^{30}]$, are used in our implementation of the BB algorithms and their stabilized versions.
\begin{figure}[t!]
	\centering
	\vspace{-6.7cm}
	\includegraphics[width=0.99\textwidth]{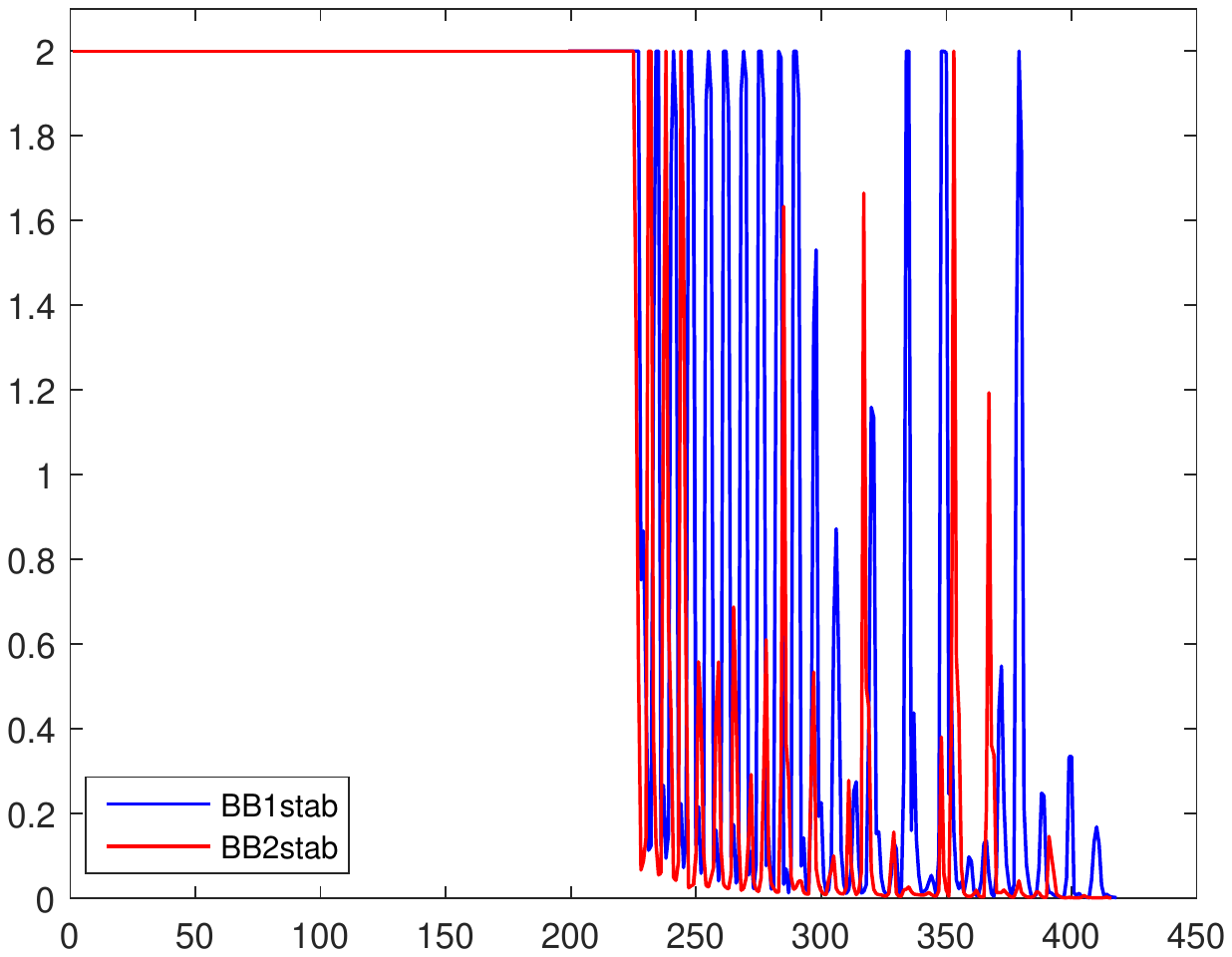}
	\vspace{-6.3cm}
	\caption{Graphs of $\|s_k\|$ for BB1stab and BB2stab with $\Delta = 2$ for Raydan function \eqref{raydan}.}\label{fig:Raydan_stab}
\end{figure}
Figure~\ref{fig:Raydan_stab} illustrates the stabilization effect. One can see that the BB1 was generating too long steps more frequently than the BB2. This observation is in general agreement with the other numerical experiments that we performed and also with the theory, which says that
$\alpha_k^{BB1} \ge \alpha_k^{BB2}$.
\begin{figure}[b!]
	\centering
	\vspace{-6.7cm}
	\includegraphics[width=0.99\textwidth]{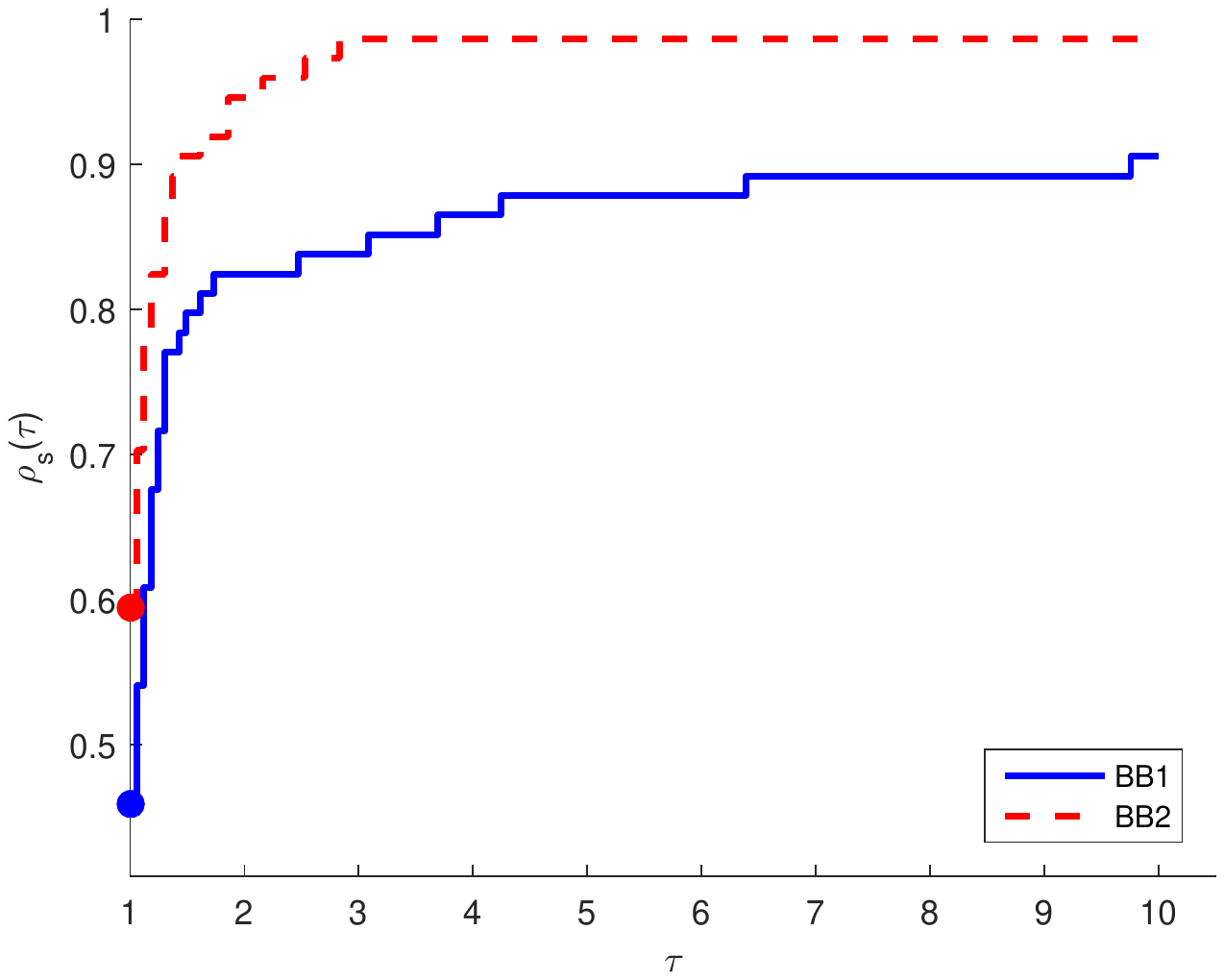}
	\vspace{-6.3cm}
	\caption{Performance profiles of the BB1 and BB2 algorithms adapted to solving nonconvex unconstrained minimization problems (based on solving problems from the CUTEst collection).}\label{BB1vsBB2}
\end{figure}

The performance of our algorithms was compared also for unconstrained minimization problems from the CUTEst collection \cite{cutest}, which provides a standard starting point $x_0$ for each of them.
We excluded from our comparison quadratic problems and those, in which the BB1/BB2 algorithm converged in less than $20$ iterations. The results reported here concern only the problems, where at least one of the compared algorithms converged, and also those, where the both algorithms converged to the same point.

Recall that the BB method was originally designed for solving convex problems in which case it is guaranteed that $\alpha_k^{BB}$ is nonnegative. Since the most of the unconstrained minimization test problems in the CUTEst collection are nonconvex, we had to adapt the BB method to solving this kind of problems.
In our implementation of the BB method and its stabilized version, we follow paper \cite{dai2015positive} in setting
\begin{equation}\label{positive_alpha}
\alpha_k^{BB} \leftarrow \dfrac{\|s_k\|}{\|y_k\|},
\end{equation}
whenever $\alpha_k^{BB} \le 0$. This makes our algorithms much more robust.
Figure \ref{BB1vsBB2} presents results of solving $74$ problems from the CUTEst collection. The BB1 and BB2 algorithms failed in $4$ and $3$ cases, respectively. The plots of the performance profiles introduced in \cite{profiles} indicate that the BB2 algorithm is more robust than the BB1. Furthermore, the former algorithm required, on average, fewer iterations for solving problems. The main reason is that the BB1 algorithm generates too long steps more frequently.
In what follows, we focus on presenting here results of stabilizing the BB1 algorithm, because it gains more from the stabilization than the BB2 algorithm.

\begin{table}[t!]
	\begin{center}
		\caption{Numerical results for unconstrained minimization problems from the CUTEst collection, adaptive selection of $\Delta$.
		}
		\label{tab:nonlin_c}
		\begin{scriptsize}
			\begin{tabular}{|l r r r c|l r r r c|}
	\hline
	\multicolumn{2}{|c}{PROBLEM} & \multicolumn{1}{c}{BB1} & \multicolumn{2}{c|}{BB1stab} & \multicolumn{2}{|c}{PROBLEM} & \multicolumn{1}{c}{BB1} & \multicolumn{2}{c|}{BB1stab}
	\\
	\multicolumn{1}{|c}{name} & \multicolumn{1}{c|}{$n$} & \multicolumn{1}{|c|}{it} & \multicolumn{1}{|c}{it} & \multicolumn{1}{c|}{$c$} & \multicolumn{1}{|c}{name} & \multicolumn{1}{c|}{$n$} & \multicolumn{1}{|c|}{it} & \multicolumn{1}{|c}{it} & \multicolumn{1}{c|}{$c$}
				\\
				\hline
				ALLINITU	&	4	&	21	&	\bf 21	&	0.1	&	EXTROSNB	&	1\,000	&	23	&	\bf 23		&	0.5	\\
				ARGTRIGLS	&	200	&	626	&	\bf 626	&	0.5	&	FLETCBV2	&	5\,000	&	30\,225	&	98\,735		&	1	\\
				BA-L1LS	&	57	&	34	&	\bf 33	&	1	&	FLETCHCR	&	1\,000	&	1\,892	&	1\,964		&	1	\\
				BA-L16LS	&	66\,462	&	64	&	66	&	0.5	&	FREUROTH	&	5\,000	&	52	&	\bf 52		&	0.5	\\
				BA-L21LS	&	34\,134	&	197	&	\bf 179	&	1	&	HEART8LS	&	8	&	44	&	\bf 44		&	0.5	\\
				BA-L49LS	&	23\,769	&	65	&	\bf 60	&	1	&	HYDC20LS	&	99	&	35	&	\bf 35		&	1	\\
				BA-L52LS	&	192\,627	&	280	&	\bf 277	&	0.1	&	LUKSAN11LS	&	100	&	31	&	32		&	1	\\
				BA-L73LS	&	33\,753	&	65	&	69	&	0.5	&	LUKSAN12LS	&	98	&	40	&	\bf 38		&	0.5	\\
				BDQRTIC	&	5000	&	41	&	\bf 41	&	0.5	&	LUKSAN17LS	&	100	&	230	&	\bf 187		&	0.1	\\
				BROWNBS	&	2	&	4\,110	&	\bf 961	&	0.1	&	LUKSAN21LS	&	100	&	6\,284	&	28\,255		&	1	\\
				BROYDN3DLS	&	5\,000	&	21	&	\bf 21	&	0.1	&	LUKSAN22LS	&	100	&	64	&	\bf 51		&	0.1	\\
				BROYDN7D	&	5\,000	&	29	&	\bf 29	&	0.1	&	MOREBV	&	5\,000	&	54\,926	&	$>10^5$		&	1	\\
				BROYDNBDLS	&	5\,000	&	58	&	\bf 58	&	1	&	MSQRTALS	&	1\,024	&	71	&	\bf 56		&	0.5	\\
				CHAINWOO	&	4\,000	&	96	&	\bf 42	&	1	&	MSQRTBLS	&	1\,024	&	56	&	59		&	0.5	\\
				CHNROSNB	&	50	&	133	&	\bf 133	&	0.5	&	NCB20	&	5\,010	&	23	&	\bf 22		&	0.1	\\
				CHNRSNBM	&	50	&	93	&	\bf 93	&	0.1	&	NONDQUAR	&	5\,000	&	40\,401	&	89\,179		&	1	\\
				CRAGGLVY	&	5\,000	&	56	&	\bf 50	&	1	&	NONMSQRT	&	4\,900	&	54	&	\bf 54		&	0.1	\\
				CUBE	&	2	&	$>10^5$	&	\bf 61	&	1	&	OSCIGRAD	&	100\,000	&	81	&	\bf 81		&	1	\\
				CURLY10	&	10\,000	&	64	&	\bf 56	&	0.1	&	OSCIPATH	&	10	&	30	&	\bf 30		&	0.1	\\
				CURLY20	&	10\,000	&	56	&	\bf 56	&	0.5	&	PENALTY2	&	200	&	730	&	1\,909		&	1	\\
				CURLY30	&	10\,000	&	57	&	\bf 57	&	0.5	&	PENALTY3	&	200	&	21	&	\bf 21		&	0.5	\\
				DENSCHNF	&	2	&	122	&	\bf 31	&	0.5	&	POWELLSG	&	5\,000	&	65	&	\bf 47		&	1	\\
				DIXMAANE	&	3\,000	&	24	&	\bf 23	&	0.1	&	ROSENBR	&	2	&	$>10^5$	&	\bf 332		&	1	\\
				DIXMAANF	&	3\,000	&	24	&	\bf 23	&	0.1	&	ROSENBRTU	&	2	&	$>10^5$	&	\bf 85		&	1	\\
				DIXMAANI	&	3\,000	&	22	&	\bf 22	&	1	&	SCURLY30	&	10\,000	&	252	&	\bf 234		&	0.1	\\
				DIXMAANJ	&	3\,000	&	23	&	\bf 23	&	1	&	SPMSRTLS	&	4\,999	&	335	&	\bf 268		&	0.1	\\
				DIXMAANM	&	3\,000	&	773	&	\bf 515	&	1	&	SROSENBR	&	5\,000	&	$>10^5$	&	\bf 55		&	0.5	\\
				DIXMAANN	&	3\,000	&	711	&	\bf 502	&	0.5	&	SSBRYBND	&	5\,000	&	4\,247	&	11\,005		&	1	\\
				DIXMAANO	&	3\,000	&	589	&	\bf 417	&	1	&	SSCOSINE	&	5\,000	&	3\,882	&	10\,414		&	1	\\
				DIXMAANP	&	3\,000	&	310	&	\bf 305	&	0.1	&	TOINTGOR	&	50	&	40	&	44		&	0.5	\\
				EDENSCH	&	2\,000	&	48	&	\bf 36	&	1	&	TOINTGSS	&	5\,000	&	5\,006	&	\bf 5\,004		&	1	\\
				EIGENALS	&	2\,550	&	41	&	\bf 41	&	0.1	&	VAREIGVL	&	50	&	415	&	\bf 323		&	0.5	\\
				EIGENCLS	&	2\,652	&	145	&	170	&	0.5	&	VESUVIALS	&	8	&	235	&	$>10^5$		&	1	\\
				ERRINROS	&	50	&	2\,920	&	\bf 746	&	1	&	VESUVIOULS	&	8	&	256	&	$>10^5$		&	1	\\
				ERRINRSM	&	50	&	25\,807	&	\bf 7\,366	&	0.1	&	WATSON	&	12	&	120	&	217		&	0.1	\\
				
				\hline
			\end{tabular}
		\end{scriptsize}
	\end{center}
\end{table}
Table \ref{tab:nonlin_c} presents results of solving $70$ nonquadratic test problems from the CUTEst collection. We tried only three values of the parameter $c$ in the adaptive choice of $\Delta$ using \eqref{delta}, namely, $0.1$, $0.5$ and $1.0$. The BB1 and BB1stab algorithms were not able to solve problems during $10^5$ iterations in $4$ and $3$ cases, respectively. The BB1stab requires fewer number of iterations in $32$ cases, while the BB1 performs better only in $17$ cases. In $21$ cases, the BB1stab with the indicated values of $c$ requires the same number of iterations as the BB1.

We made experiments also with directly setting a certain value of $\Delta$ in the BB1stab. The trial values were $0.01$, $0.1$ and $1.0$.
For a few test problems, the results are better than for the aforementioned adaptive choice with $c = 0.1$, $0.5$ and $1.0$. For $22$ of $71$ problems, the number of iterations is smaller than in case of the BB1.
\begin{table}[t!]
	\begin{center}
		\caption{Numerical results for unconstrained minimization problems from the CUTEst collection, preselected $\Delta$.
		}
		\label{tab:nonlin_delta}
		\begin{scriptsize}
			\begin{tabular}{|l r r r c|l r r r c|}
				\hline
				\multicolumn{2}{|c}{PROBLEM} & \multicolumn{1}{c}{BB1} & \multicolumn{2}{c|}{BB1stab} & \multicolumn{2}{|c}{PROBLEM} & \multicolumn{1}{c}{BB1} & \multicolumn{2}{c|}{BB1stab}
				\\
				\multicolumn{1}{|c}{name} & \multicolumn{1}{c|}{$n$} & \multicolumn{1}{|c|}{it} & \multicolumn{1}{|c}{it} & \multicolumn{1}{c|}{$\Delta$} & \multicolumn{1}{|c}{name} & \multicolumn{1}{c|}{$n$} & \multicolumn{1}{|c|}{it} & \multicolumn{1}{|c}{it} & \multicolumn{1}{c|}{$\Delta$}
				\\
				\hline
				BROWNBS	&	2	&	4\,110	&	\bf 80	&	1	&	LUKSAN11LS	&	100	&	31	&	\bf 23		&	1	\\
				CHNROSNB	&	50	&	133	&	\bf 50	&	1	&	LUKSAN17LS	&	100	&	230	&	\bf 166		&	1	\\
				CHNRSNBM	&	50	&	93	&	\bf 41	&	1	&	MOREBV	&	5\,000	&	54\,926	&	\bf 44\,712		&	0.01	\\
				CUBE	&	2	&	$>10^5$	&	\bf 94	&	0.1	&	MSQRTALS	&	1\,024	&	71	&	\bf 66		&	0.1	\\
				DENSCHNF	&	2	&	122	&	\bf 31	&	1	&	NONMSQRT	&	4\,900	&	54	&	\bf 51		&	1	\\
				DIXMAANM	&	3\,000	&	773	&	\bf 715	&	1	&	OSCIPATH	&	10	&	30	&	\bf 27		&	1	\\
				DIXMAANO	&	3\,000	&	589	&	\bf 514	&	1	&	ROSENBR	&	2	&	$>10^5$	&	\bf 129		&	0.1	\\
				ERRINROS	&	50	&	2\,920	&	\bf 923	&	1	&	ROSENBRTU	&	2	&	$>10^5$	&	\bf 664		&	0.1	\\
				ERRINRSM	&	50	&	25\,807	&	\bf 6\,165	&	0.1	&	SPMSRTLS	&	4\,999	&	335	&	\bf 294		&	1	\\
				FLETCBV2	&	5\,000	&	30\,225	&	\bf 25\,325	&	1	&	SROSENBR	&	5\,000	&	$>10^5$	&	\bf 206		&	1	\\
				FLETCHCR	&	1\,000	&	1\,892	&	\bf 572	&	1	&	TQUARTIC	&	5\,000	&	F	&	\bf 5\,325		&	0.1	\\
				\hline
			\end{tabular}
		\end{scriptsize}
	\end{center}
\end{table}
These results are reported in Table \ref{tab:nonlin_delta}. The preselected values of $\Delta$ allowed the BB1stab to solve five problems of those not solved by the BB1, including problems MOREBV and TQUARTIC, in which the adaptive choice of $\Delta$ failed. In case of TQUARTIC, the BB1 terminated because of producing NaN (Not a Number) in Matlab. The experiments with the preselected values of $\Delta$ indicate that there is plenty of room for improving the very simple adaptive strategy proposed in this paper.

\begin{table}[thb]
	\begin{center}
		\caption{Numerical results for unconstrained minimization problems from the CUTEst collection, adaptive selection of $\Delta$.
		}
		\label{tab:nonlin_c_bb2}
		\begin{scriptsize}
			\begin{tabular}{|l r r r c|l r r r c|}
				\hline
				\multicolumn{2}{|c}{PROBLEM} & \multicolumn{1}{c}{BB2} & \multicolumn{2}{c|}{BB2stab} & \multicolumn{2}{|c}{PROBLEM} & \multicolumn{1}{c}{BB2} & \multicolumn{2}{c|}{BB2stab}
				\\
				\multicolumn{1}{|c}{name} & \multicolumn{1}{c|}{$n$} & \multicolumn{1}{|c|}{it} & \multicolumn{1}{|c}{it} & \multicolumn{1}{c|}{$c$} & \multicolumn{1}{|c}{name} & \multicolumn{1}{c|}{$n$} & \multicolumn{1}{|c|}{it} & \multicolumn{1}{|c}{it} & \multicolumn{1}{c|}{$c$}
				\\
				\hline
				BA-L21LS	&	34134	&	191	&	\bf	187		&	1	&	EIGENALS	&	2550	&	44	&	\bf	39			&	0.5	\\
				BA-L52LS	&	192627	&	358	&	\bf	316		&	1	&	EIGENCLS	&	2652	&	201	&	\bf	177			&	0.1	\\
				BDQRTIC	&	5000	&	41	&	\bf	37		&	0.1	&	INDEFM	&	100000	&	23	&	\bf	21			&	1	\\
				BROWNBS	&	2	&	4110	&	\bf	961		&	0.1	&	LUKSAN17LS	&	100	&	198	&	\bf	183			&	0.1	\\
				CHNRSNBM	&	50	&	54	&	\bf	45		&	0.1	&	MSQRTALS	&	1024	&	77	&	\bf	72			&	0.5	\\
				DENSCHNF	&	2	&	29	&	\bf	28		&	1	&	MSQRTBLS	&	1024	&	59	&	\bf	58			&	0.5	\\
				DIXMAANE	&	3000	&	21	&	\bf	20		&	0.1	&	NONDIA	&	5000	&	-	&	\bf	10599			&	0.5	\\
				DIXMAANF	&	3000	&	24	&	\bf	22		&	0.1	&	OSCIPATH	&	10	&	26	&	\bf	25			&	0.1	\\
				DIXMAANI	&	3000	&	25	&	\bf	20		&	0.1	&	PENALTY3	&	200	&	22	&	\bf	21			&	0.1	\\
				DIXMAANJ	&	3000	&	24	&	\bf	22		&	0.1	&	POWELLSG	&	5000	&	44	&	\bf	38			&	0.5	\\
				DIXMAANM	&	3000	&	610	&	\bf	425		&	0.1	&	VAREIGVL	&	50	&	490	&	\bf	407			&	0.1	\\
				DIXMAANN	&	3000	&	611	&	\bf	448		&	0.1	&	WATSON	&	12	&	340	&	\bf	170			&	0.5	\\
				DIXMAANO	&	3000	&	464	&	\bf	414		&	0.1	&		&		&		&					&		\\
				\hline
			\end{tabular}
		\end{scriptsize}
	\end{center}
\end{table}

For the BB2stab algorithm, we still tried the same three values of the parameter $c$ in the adaptive choice of $\Delta$ using \eqref{delta} as for BB1stab. In $77$ test problems, the BB2stab performs better in $25$ cases, while the BB2 performs better only in $15$ cases. Table \ref{tab:nonlin_c_bb2} presents results for all the cases when the BB2stab requires fewer number of iterations.

\section{Conclusions}
In the present paper, it was proposed to stabilize the conventional BB method by virtue of bounding the distance between sequential iterates. The purpose was to improve its convergence, when it is affected by too long steps $\|\alpha_k^{BB} g_k\|$, and also to make the BB method convergent, when it fails to converge. Both a theoretical and numerical study of the stabilized version was conducted. We have proved that the stabilization provides the BB method with a global convergence without recourse to using any line search. The numerical results presented here are highly encouraging. The proposed very simple adaptive selection of $\Delta$ was able to successfully trap a value which is appropriate for each specific problem. However, we hope that this paper will stimulate development of more efficient algorithms for adaptive selection of $\Delta$.

\bigskip
\noindent {\bf Acknowledgments.} Part of this work was done during Oleg Burdakov’s visit to
the Chinese Academy of Sciences, which was supported by the Visiting Scientist award under the Chinese Academy of Sciences President's International Fellowship Initiative for 2017.
The second author was supported by the Chinese Natural Science Foundation (No. 11631013)
and the National 973 Program of China (No. 2015CB856002).


\begin{thebibliography}{10}
	
\bibitem{Barzilai1988two}
	J.~Barzilai and J.~M. Borwein,
	\newblock Two-point step size gradient methods,
	\newblock {\em IMA J. Numer. Anal.}, \textbf{8}:1 (1988), 141-148.
	
\bibitem{birgin2000nonmonotone}
	E.~G. Birgin, J.~M. Mart{\'\i}nez and M.~Raydan,
	\newblock Nonmonotone spectral projected gradient methods on convex sets,
	\newblock {\em SIAM J. Optim.}, \textbf{10}:4 (2000), 1196-1211.
	
\bibitem{Birgin2014review}
	E.~G. Birgin, J.~M. Mart{\'\i}nez and M.~Raydan,
	Spectral projected gradient methods: Review and perspectives,
	\emph{Journal of Statistical Software}, \textbf{60}:3 (2014), 1-21.

\bibitem{CurtisGuo18}
    F.~E. Curtis and W. Guo,
    R-linear convergence of limited memory steepest descent,
    \emph{IMA Journal of Numerical Analysis}, \textbf{38} (2018), 720-742.

\bibitem{dai2003}
	Y.-H. Dai,
	Alternate step gradient method,
	{\em Optimization}, \textbf{52} (2003), 395-415.
	
\bibitem{dai2015positive}
	Y.-H. Dai, M.~Al-Baali and X.~Yang,
	\newblock A positive Barzilai-Borwein-like stepsize and an extension for
	symmetric linear systems,
	\newblock in {\em Numerical Analysis and Optimization}, Springer, 2015, 59-75.

\bibitem{dai2005asymptotic}
	Y.-H. Dai and R.~Fletcher,
	\newblock On the asymptotic behaviour of some new gradient methods,
	\newblock {\em Math. Program.}, \textbf{103} (2005), 541-559.
	
\bibitem{dai2005projected}
	Y.-H. Dai and R.~Fletcher,
	\newblock Projected Barzilai-Borwein methods for large-scale box-constrained
	quadratic programming,
	\newblock {\em Numer. Math.}, \textbf{100}:1 (2005), 21-47.
	
\bibitem{dai2006cyclic}
	Y.-H. Dai, W.~W. Hager, K.~Schittkowski and H.~Zhang,
	\newblock The cyclic Barzilai-Borwein method for unconstrained optimization,
	\newblock {\em IMA J. Numer. Anal.}, \textbf{26}:3 (2006), 604-627.
	
\bibitem{dai2002r}
	Y.-H. Dai and L.-Z. Liao,
	\newblock R-linear convergence of the Barzilai and Borwein gradient method,
	\newblock {\em IMA J. Numer. Anal.}, \textbf{22}:1 (2002), 1-10.
	
\bibitem{dai2005analysis}
	Y.-H. Dai, L.-Z. Liao and D.~Li,
	\newblock An analysis of the Barzilai and Borwein gradient method for
	unsymmetric linear equations,
	\newblock In {\em Optimization and Control with Applications},
	Springer, 2005, 183-211.

\bibitem{DavisHu-11}
	T.A. Davis and Y. Hu, The University of Florida sparse matrix collection,
	\emph{ACM Transactions on Mathematical Software}, \textbf{38}:1 (2011), 1-25.

\bibitem{profiles}
	E.~D. Dolan and J.J. Mor\'e, Benchmarking optimization software
	with performance profiles, \textit{Math. Programming}, \textbf{91} (2002), 201-213.
	
\bibitem{fletcher2005barzilai}
	R.~Fletcher,
	\newblock On the Barzilai-Borwein method,
	\newblock in {\em Optimization and control with applications}, Springer,
	2005, 235-256.
	
\bibitem{friedlander1998gradient}
	A.~Friedlander, J.~M. Mart{\'\i}nez, B.~Molina and M.~Raydan,
	\newblock Gradient method with retards and generalizations,
	\newblock {\em SIAM J. Numer. Anal.}, \textbf{36}:1 (1998), 275-289.

\bibitem{cutest}
	N.I.M.~Gould, D.Orban and Ph.L.~Toint,
	CUTEst: a constrained and unconstrained testing environment with safe
	threads for mathematical optimization,
	\textit{Computational Optimization and Applications}, \textbf{60} (2015), 545-557.

\bibitem{Grippo_etal_86}
	L.~Grippo, F.~Lampariello and S.~Lucidi,
	A nonmonotone line search technique for Newton's method,
	\emph{SIAM J. Numer. Anal.}, \textbf{23}:4, 707-716.

\bibitem{GrippoSciandrone2002}
	L.~Grippo and M.~Sciandrone,
	Nonmonotone globalization techniques for the Barzilai-Borwein gradient method,
	\emph{Computational Optimization and Applications}, \textbf{23}:2 (2002), 143-169.

\bibitem{LiuDai2001}
	W.~Liu and Y.-H.~Dai
	Minimization algorithms based on
	supervisor and searcher cooperation,
	\textit{J. Optim. Theory Appl.}, \textbf{111}:2 (2001), 359-379.
		
\bibitem{raydan1993Barzilai}
	M.~Raydan,
	\newblock On the Barzilai and Borwein choice of steplength for the gradient
	method,
	\newblock {\em IMA J. Numer. Anal.}, \textbf{13}:3 (1993), 321-326.
	
\bibitem{raydan1997barzilai}
	M.~Raydan,
	\newblock The Barzilai and Borwein gradient method for the large scale
	unconstrained minimization problem,
	\newblock {\em SIAM J. Optim.}, \textbf{7}:1 (1997), 26-33.

\bibitem{sparse}
	\textit{The SuiteSparse Matrix Collection},
	https://sparse.tamu.edu/, 2019.
	
\bibitem{yuan2008step}
	Y.-X. Yuan,
	\newblock Step-sizes for the gradient method,
	\newblock {\em AMS IP Studies in Advanced Mathematics}, \textbf{42}:2 (2008), 785-796.
	
\end{thebibliography}
\end{document}